\documentclass[11pt]{article}
\usepackage{amsmath,amssymb,amsthm}
\usepackage{graphicx}
\usepackage{enumerate}
\usepackage{geometry,graphicx}
\usepackage[x11names,usenames,dvipsnames,svgnames,table]{xcolor}
\usepackage{a4wide}
\usepackage{pdfsync}
\usepackage{epstopdf}
\usepackage{comment}
\numberwithin{equation}{section}
\newtheorem{theorem}{Theorem}[section]

\newtheorem{corollary}{Corollary}[section]

\newtheorem{lemma}{Lemma}[section]
\theoremstyle{definition}
\newtheorem{remark}{Remark}[section]
\newtheorem{definition}{Definition}[section]
\newtheorem{example}{Example}[section]
\newtheorem{assumption}{Assumption}

\newcommand{\ie}{{\em i.e.}, }
\newcommand{\eg}{{\em e.g.}, }
\newcommand{\cf}{{\em cf.\ }}

\newcommand{\nn}{\mathbb{N}} %Nonnegative integers
\newcommand{\norm}[1]{\left\Vert {#1} \right\Vert} %Norm
\newcommand{\erl}{\left(-\infty , +\infty\right]} %Extended real line
\newcommand{\ridom}[1]{\mathrm{ri(dom\,{#1}})} %Relative interior of the Domain
\newcommand{\argmin}{\operatorname{argmin}}

\newcommand{\act}[1]{\left\langle {#1} \right\rangle} %The value of 1
\newcommand{\seq}[2]{\left\{{#1}_{{#2}}\right\}_{{#2} \in \mathbb{N}}}
\newcommand{\Seq}[2]{\left\{{#1}^{{#2}}\right\}_{{#2} \in \mathbb{N}}}

\newcommand{\Lag}{\mathcal{L}} % Lagrangian
\newcommand{\AAA}{\mathcal{A}}

\newcommand{\FFF}{\mathcal{F}}

\newcommand{\PPP}{\mathcal{P}}

\newcommand{\bo}{{\bf 0}}

\newcommand{\bbS}{\mathbb{S}}
\newcommand{\real}{\mathbb{R}} %Real numbers
\newcommand{\rr}{\mathbb{R}} %Real numbers
\newcommand{\nice}{nice\,\,}
\newcommand{\SSS}{\operatorname{Prim}}

\def\bl{\color{blue}}
\title{Faster Lagrangian-Based Methods in Convex Optimization}
\author{Shoham Sabach\footnote{Faculty of Industrial Engineering and Management, The Technion, Haifa, 32000, Israel. E-mail: ssabach@ie.technion.ac.il.}\and Marc Teboulle\footnote{School of Mathematical Sciences, Tel-Aviv University, Ramat-Aviv 69978, Israel. E-mail: teboulle@post.tau.ac.il. This research was partially supported by the Israel Science Foundation, under ISF Grants 1844-16, and 2619-20.}}
\date{\today}

\begin{document}
\maketitle

\begin{abstract}
In this paper, we aim at unifying, simplifying and improving the convergence rate analysis of Lagrangian-based methods for convex optimization problems. We first introduce the notion of nice primal algorithmic map, which plays a central role in the unification and in the simplification of the analysis of most Lagrangian-based methods. Equipped with a \nice primal algorithmic map, we then introduce a versatile generic scheme, which allows for the design and analysis of Faster LAGrangian (FLAG) methods with new provably sublinear rate of convergence expressed in terms of function values and feasibility violation of the original (non-ergodic) generated sequence. To demonstrate the power and versatility of our approach and results, we show that most well-known iconic Lagrangian-based schemes admit a \nice primal algorithmic map, and hence share the new faster rate of convergence results within their corresponding FLAG.
\end{abstract}



\section{Introduction} \label{Sec:Intro}
	Over the last decade, we are witnessing a resurgence of Lagrangian-based methods. This surge of interest is due to the flexibility that they provide in exploiting special problem structures that abound in modern and disparate applications (\eg image science, machine learning, communication systems). Augmented Lagrangian methods (also known as multiplier methods) which are a typical prototype, were invented about half a century ago by Hestenes \cite{H1969} and Powell \cite{P1969}, are the core of the analysis and development of many extensions of Lagrangian-based schemes capable of tackling optimization problems with various types of structures and constraints. This is reflected by a rich and very large volume of literature on various practical and theoretical aspects of many Lagrangian-based methods. Two classical references providing the central developments and ideas underlying Lagrangian-based and related decomposition methods are the books of Bertsekas \cite{B82-B}, and Bertsekas and Tsitsiklis \cite{BT1989-B}, which include a comprehensive list of relevant references for these developments until the nineties. Here, one should also mention the pioneer works of Rockafeller \cite{R1976,R1976:1} on the convergence analysis of proximal AL methods. For more modern developments, results, and applications, see \eg the more recent works \cite{BPCPE2011, ST2014},  \cite{LM2016,CP2016-B, ST2019}, and references therein.
\medskip

	A recent and voluminous literature on Lagrangian-based methods has focused on rate of convergence results for the so-called Alternating Direction of Multipliers Method \cite{GM75,GM1976,FG1983-B,EB92}, with an almost absolute predominance on {\em ergodic} type of rate of convergence results. Since the literature on this topic is very wide, let us just mention  the first central works \cite{CP2011,HY2012,MS2013}, which have obtained an ergodic rate of $O(1/N)$. For many more rate of convergence results of the ergodic type, and relevant references, we refer the readers to, \eg \cite{CP2016}, and the very recent monograph \cite{LLF2020}, which also includes an up to date extensive bibliography.
\medskip

	The focus of this paper is on faster and non-ergodic global rate of convergence analysis for Lagrangian-based methods applied to convex optimization, where the rate is measured in terms of function values and feasibility violation of the original produced sequence. The literature on papers obtaining non-ergodic rate of convergence results remains very limited, in particular, in term of the measures just mentioned. For instance, in \cite{HY2015}, the authors obtained a non-ergodic $O(1/N)$ result for the squared distance between two successive primal iterates, while in \cite{CP2011}, a non-ergodic $O(1/N^{2})$ result, in the strongly convex setting, is proven for the squared distance between the primal iterates to an optimal primal solution. One exception is the very recent work \cite{LL2019}, which derives a new non-ergodic $O(1/N)$ rate of convergence result in terms of function values and feasibility violation of the original produced sequence for the specific Linearized ADMM, through a rather involved and nontrivial analysis (compare with the results developed in Section \ref{Sec:Examples} for this particular algorithm).
\medskip
	
	The analysis of Lagrangian-based methods is usually intricate and complicated, which results in lengthy proofs. Recently, in \cite{ST2019}, we have shown that most well-known Lagrangian-based methods and their decomposition variants can be captured and analysed through a single scheme (that we called the Perturbed Proximal Method of Multipliers, see details in \cite{ST2019}). It allows us to derive an ergodic $O(1/N)$  rate of convergence in terms of function values and feasibility violation. Much like most of the existing literature, the approach we developed there was limited to Lagrangian-based methods with classical sublinear ergodic type results, and hence naturally leads to the question: {\em Can we develop an algorithmic framework capable of producing non-ergodic rate as well as faster Lagrangian-based schemes?} In this paper, we provide a positive answer, within a surprisingly simple and elegant approach. This is developed in four main sections which contents are now briefly outlined.
\medskip	
	
	The formulation of the problem under study, some basic preliminaries on Lagrangians in the convex setting, and some important examples of convex optimization problems which fit our model, can be found in Section \ref{Sec:Form}. The starting point of our study is that most Lagrangian-based methods can be simply described along the following protocol. Consider the corresponding augmented Lagrangian of the problem at hand, which consists of primal and dual (multiplier) variables, building a Lagrangian-based method requires to choose updating rules for both the primal and the dual variables. Since, in most Lagrangian-based methods, the dual variable is updated via an explicit formula, the main difference between Lagrangian-based methods is encapsulated in the choice of a {\em primal algorithmic map} which updates the primal variables. This choice is governed by the problem's structure and data information, which must be best exploited, and can be seen as a single step of any appropriate optimization method that applied on the augmented Lagrangian itself, or a variation of it. For instance, the classical Augmented Lagrangian method \cite{H1969,P1969}, is a typical example whereby the primal algorithmic map is nothing else but an {\em exact minimization} applied on the augmented Lagrangian itself associated to problem (P).
\medskip

	Adopting this point of view for Lagrangian-based methods, we introduce in Section \ref{Sec:GL} the notion of {\em nice primal algorithmic map}, which will play a central role in unifying the analysis of most Lagrangian-based methods into a unitary and simple framework. In turns, equipped with a \nice primal algorithmic map, we propose a generic algorithmic framework, called FLAG for Faster LAGrangian based methods that aims at achieving faster (as well as classical) convergence rates of most Lagrangian-based methods {\em at once}. FLAG which is equipped with a \nice primal algorithmic maps allows for the design and analysis of faster schemes with new provable sublinear {\em Non-ergodic} rate of $O(1/N^{2})$ in the strongly convex case and a {\em Non-ergodic} $O(1/N)$ in the convex case (see Theorem \ref{T:FGLrate} and Theorem \ref{T:GLrate}, respectively). As a by product, we also derive classical and faster ergodic results in Corollaries \ref{C:FGLrate} and \ref{C:GLrate}; see Section \ref{Sec:AnalysisGL}. To illustrate the power and versatility of FLAG and our new results, in Section \ref{Sec:Examples}, we consider various block and other models of problem (P) and demonstrate that most well-known iconic schemes, \eg Linearized Proximal Method of Multipliers, ADMM and its linearized versions, Jacobi type decomposition, Predictor Corrector Proximal Multipliers, and many more, all admit a \nice primal algorithmic map, and hence can be used in FLAG to obtain new rate of convergence results. This provides a significant list of faster schemes  for fundamental well-known Lagrangian-based methods.
		
\section{Problem Formulation} \label{Sec:Form}
	In this paper, we consider the linearly constrained convex optimization model,
	\begin{equation*}
		\mbox{(P)} \quad\quad \min  \left\{ \Psi\left(x\right) \, : \, \AAA x = b, \; x \in \real^n\right\},
	\end{equation*}
	where $\Psi:\real^{n} \rightarrow \erl$ is a proper, lower semi-continuous and convex function, $\AAA: \real^n\to \real^m$ is a linear mapping, and $b \in \real^{m}$. Throughout this work, the feasible set of problem (P) is denoted by  $\FFF = \left\{ x \in \rr^{n} : \, \AAA x = b \right\}$.
\medskip
	
	 Despite its apparent simplicity, with an additional structural/data information on the objective function or/and on the linear mapping (see examples below), the optimization model (P) essentially captures a very wide range of convex optimization problems arising in many modern applications. Since we are interested in analyzing this model also in the strongly convex case, we instead assume that $\Psi\left(\cdot\right)$ is $\sigma$-strongly convex with $\sigma \geq 0$ (\ie when $\sigma = 0$ we recover the convex setting). At this juncture, it is important to mention that in many applied problems the function $\Psi$ is defined by the sum of two or more functions, see \eg \cite{CP2016-B}. In such cases,  it is enough to assume that {\em one} of the functions is $\sigma$-strongly convex, see Section \ref{SSec:Block} for more details.
\medskip
	
	Let us illustrate a few basic and important instances of model (P).
	\begin{example}[Linear composite model] \label{E:LCM}
		The linear composite optimization model, which naturally appears in many applications, fuels very much the study of Lagrangian-based methods and is given by
		\begin{equation} \label{M:LC}
			\min_{u \in \real^{p}} \left\{ f\left(u\right) + g\left(Au\right) \right\},
		\end{equation}
		where $f : \real^{p} \rightarrow \erl$ and $g : \real^{q} \rightarrow \erl$ are proper, lower semi-continuous and convex functions, and $A : \real^{p} \rightarrow \real^{q}$ is a linear mapping. Rewriting problem \eqref{M:LC} as an equivalent linearly constrained optimization problem (which requires the definition of an auxiliary variable) yields
		\begin{equation*}
			\min_{u \in \real^{p} , v \in \real^{q}} \left\{ f\left(u\right) + g\left(v\right) : \, Au = v \right\}.
		\end{equation*}
		We easily see that it fits into model (P), with $x = \left(u^{T} , v^{T}\right)^{T}$, $\Psi\left(x\right) := f\left(u\right) + g\left(v\right)$ and $\AAA x = Au - v$ (\ie $b = \bo$).
	\end{example}
	\begin{example}[Block linear constrained model] \label{E:BLCM}
		Another useful and more general model which includes \eqref{M:LC}  is the following
		\begin{equation} \label{M:BLC}
			\min_{u \in \real^{p} , v \in \real^{q}} \left\{ f\left(u\right) + g\left(v\right) : \, Au + Bv = b \right\},
		\end{equation}
		where $f : \real^{p} \rightarrow \erl$ and $g : \real^{q} \rightarrow \erl$ are proper, lower semi-continuous and convex functions, $A : \real^{p} \rightarrow \real^{m}$ and $B : \real^{q} \rightarrow \real^{m}$ are linear mappings. Again we can easily see that it fits into model (P), with $x = \left(u^{T} , v^{T}\right)^{T}$, $\Psi\left(x\right) := f\left(u\right) + g\left(v\right)$ and $\AAA x = Au + Bv$.
	\end{example}	
	\begin{example}[Additive smooth/non-smooth composite model] \label{E:ASNS}
		A very classical model in optimization minimizes an additive composite objective, which is the sum of two functions one which is smooth and one which is not necessarily smooth. Here we consider this model with an additional linear constraint, that is,
		\begin{equation} \label{M:AS}
			\min_{x \in \real^{n}} \left\{ f\left(x\right) + h\left(x\right) : \, \AAA x = b \right\},
		\end{equation}
		where $f : \real^{p} \rightarrow \erl$ is a proper, lower semi-continuous and convex function, while $h : \real^{q} \rightarrow \rr$ is a continuously differentiable function with a Lipschitz continuous gradient. Again we easily see that it fits into model (P), with $\Psi\left(x\right) := f\left(x\right) + h\left(x\right)$, being the sum of a smooth function and a non-smooth function. This additional structure in the objective function can be beneficially exploited when developing an adequate Lagrangian-based method, as we discuss in details in Section \ref{SSec:Smooth}.
	\end{example}			
	We conclude this section with some basic properties of Lagrangians in convex optimization. For the equality constrained problem given in model (P), the {\em Lagrangian} $\Lag: \rr^{n} \times \rr^{m} \to \erl$ is defined by
	\begin{equation} \label{D:Lag}
		\Lag\left(x , y\right) \equiv \Psi\left(x\right) + \act{y , \AAA x - b},
	\end{equation}
	where $y \in \real^{m}$ is the {\em Lagrange multiplier} associated with the linear constraint $\AAA x = b$. In this study, we will also use the \textit{augmented Lagrangian} associated to problem (P), which is defined by 	
	\begin{equation} \label{D:AugLA}
		\Lag_{\rho}\left(x ,  y\right) := \Lag\left(x , y\right) + \frac{\rho}{2}\norm{\AAA x - b}^{2} = \Psi\left(x\right) + \act{y , \AAA x - b} + \frac{\rho}{2}\norm{\AAA x - b}^{2},
	\end{equation}
	where $\rho > 0$ is a {\em penalty parameter}. Clearly, when $\rho = 0$ we recover the Lagrangian itself, \ie $\Lag \equiv \Lag_{0}$.
\medskip

	Throughout this work we make the following standing assumption.
	\begin{assumption} \label{A:ConvLag}
		The Lagrangian $\Lag$ has a saddle point, that is, there exists a pair $\left(x^{\ast} , y^{\ast}\right)$ such that
		\begin{equation} \label{A:ConvLag:1}
			\Lag\left(x^{\ast} , y\right) \leq \Lag\left(x^{\ast} , y^{\ast}\right) \leq \Lag\left(x , y^{\ast}\right), \quad \forall \,\, x \in \real^{n}, \quad \forall \,\, y\in \real^{m}.
		\end{equation}
	\end{assumption}	
	\noindent It is well-known that $\left(x^{\ast} , y^{\ast}\right)$ is a saddle point of $\Lag$ if and only if the following conditions hold (see \cite{R1970-B}):
	 \begin{itemize}
		\item[$\rm{(i)}$] $x^{\ast}$ is a solution of problem (P), \ie $\Psi\left(x^{\ast}\right)$ is finite valued and $\AAA x^{\ast} = b$.
		\item[$\rm{(ii)}$] The multiplier $y^{\ast}$ is an optimal solution of the dual problem associated to problem (P), given by
			\begin{equation*}
				\mbox{(D)} \qquad d^{\ast} = \sup_{y \in \real^{m}} \left\{ d\left(y\right) := -\Psi^{\ast}\left(-\AAA^{T}y\right) - \act{y , b} \right\},
			\end{equation*}
			and $p^{\ast} := \Lag\left(x^{\ast} , y^{\ast}\right)  = d^{\ast}$, where $\Psi^{\ast}\left(\cdot\right)$ stands for the Fenchel conjugate function of $\Psi$ (see \cite{R1970-B}).
	\end{itemize}
	Thanks to the convexity of problem (P), recall that the existence of an optimal dual Lagrange multiplier $y^{\ast}$ attached to the constraint $\AAA x = b$ is ensured under the standard constraint qualification for problem (P), that can be formulated as follows: there exists a ${\bar x} \in \ridom{\Psi}$ satisfying $\AAA{\bar x} = b$, and that strong duality holds, \ie $p^{\ast} = d^{\ast} \in \real$. Moreover, the set of optimal dual solutions is non-empty, convex and compact, see \eg \cite{R1970-B}.
\medskip
	
	We end this section with some notations that will be used throughout the paper. We denote by $\bbS_{+}^{n}$ and $\bbS_{++}^{n}$ the sets of all $n \times n$ {\em positive semi-definite matrices} and {\em positive definite matrices}, respectively. Given $P \in \bbS_{+}^{n}$, $\norm{u}_{P} := \sqrt{\act{u , Pu}}$ stands for the semi-norm of $u \in \real^{n}$. For any $u , v , w \in \real^{n}$, we define the following quantity 		
	\begin{equation} \label{D:Delta}
		\Delta_{P}\left(u , v , w\right) := \frac{1}{2}\left(\norm{u - v}_{P}^{2} - \norm{u - w}_{P}^{2}\right).
	\end{equation}
	In the case that $P = I$, where $I_{n}$ denotes the $n \times n$ {\em identity matrix}, we simply denote $\Delta := \Delta_{P}$.

\section{A Unified Framework for Lagrangian-based Methods} \label{Sec:GL}
	The starting point of our study is that most Lagrangian-based methods can be simply described along the following unified approach. Considering the corresponding augmented Lagrangian of the problem at hand, which consists of primal and dual (multiplier) variables, building a Lagrangian-based method requires to choose updating rules for both the primal and the dual variables. Since, in most Lagrangian-based methods, the dual variable is updated via an explicit formula (see \eqref{GL:2} below), the main difference between Lagrangian-based methods is encapsulated in the choice of a {\em primal algorithmic map} that updates the primal variables. Therefore, any Lagrangian-based method updates a couple $\left(z , y\right)$ by first choosing a primal algorithmic map $\PPP$, and will update the primal and dual (multiplier) variables via
	\begin{align}
		z^{+} & \in \PPP\left(z , y\right),  \label{GL:1} \\
		y^{+} & = y + \mu\rho\left(\AAA z^{+} - b\right), \label{GL:2}
	\end{align}
	where $\mu > 0$ is a scaling parameter. The update of the primal variable which depends on the algorithmic map $\PPP$, can be seen as a single step of any optimization method applied on the augmented Lagrangian itself or a variation of it. For example, the classical Augmented Lagrangian method \cite{H1969,P1969}, is a typical prototype where the mapping $\PPP$ is nothing but an {\em exact minimization} applied on the augmented Lagrangian itself.
\medskip

	Adopting this point of view of Lagrangian-based methods, we propose a general framework that aims at achieving convergence rates (classical and fast) of most Lagrangian-based methods at once. Considering the optimization problem (P) with the data $[\Psi, \AAA, b, \sigma]$ and its associated augmented Lagrangian $\Lag_{\rho}\left(\cdot\right)$, which is defined in \eqref{D:AugLA}, we introduce the notion of {\em nice primal algorithmic map} that plays a central role in unifying the analysis of most Lagrangian-based methods into a single and simple framework.
	\begin{definition}[Nice primal algorithmic map] \label{D:AlgoMapNice}
		Given the parameters $\rho , t > 0$, we let $\rho_{t} = \rho$ and $\tau_{t} = 1$ (when $\sigma = 0$) or $\rho_{t} = \rho t$ and $\tau_{t} = t$ (when $\sigma > 0$). A primal algorithmic map $\SSS_{t} : \real^{n} \times \real^{m} \rightarrow \real^{n}$, which is applied on the augmented Lagrangian $\Lag_{\rho_{t}}\left(z , \lambda\right)$, that generates $z^{+}$ via $z^{+} \in \SSS_{t}\left(z , \lambda\right)$, is called {\em \nice}\hspace{-0.05in} if there exist a parameter $\delta \in \left(0 , 1\right]$ and matrices $P , Q \in \bbS_{+}^{n}$, such that for any $\xi \in \FFF$ we have
		\begin{equation} \label{D:AlgoMapNice:1}
			\Lag_{\rho_{t}}\left(z^{+} , \lambda\right) - \Lag_{\rho_{t}}\left(\xi , \lambda\right) \leq \tau_{t}\Delta_{P}\left(\xi , z , z^{+}\right) - \frac{\tau_{t}}{2}\norm{z^{+} - z}_{Q}^{2} - \frac{\sigma}{2}\norm{\xi - z^{+}}^{2} - \frac{\delta\rho_{t}}{2}\norm{\AAA z^{+} - b}^{2}.
		\end{equation}				
	\end{definition}		
	The essence of our ability to both unify the analysis,  and to accelerate the convergence rate of most Lagrangian-based methods at once, comes from the algorithmic framework that we introduce now. This framework crucially uses as an input a \nice primal algorithmic map $\SSS_{t}\left(\cdot\right)$ with the induced parameter $\delta \in \left(0 , 1\right]$ and matrices $P , Q \in \bbS_{+}^{n}$.
\medskip

    \begin{center}
    		\fbox{\parbox{16.5cm}{{\bf FLAG -- Faster LAGrangian based method}
			\begin{enumerate}
				\item[1.] {\bf Input:} The problem data $[\Psi, \AAA, b, \sigma]$, and a \nice primal algorithmic map $\SSS_{t}\left(\cdot\right)$.
				\item[2.] {\bf Initialization:} Set $\mu \in \left(0 , \delta\right]$ and $\rho > 0$. Start with any $\left(x^{0} , z^{0} , y^{0}\right) \in \rr^{n} \times \rr^{n} \times \rr^{m}$, $t_{0} = 1$.
            		\item[3.] {\bf Iterations:} Sequences $\left\{ \left(x^{k} , z^{k} , y^{k}\right) \right\}_{k \in \nn}$ and $\seq{t}{k}$ are generated as follows: for each $k \geq 0$
                		\begin{enumerate}
                    		\item[3.1.] If $\sigma = 0$, let $\rho_{k} = \rho$, or, if $\sigma > 0$, let $\rho_{k} = \rho t_{k}$. Compute
                    			\begin{equation} \label{GL:Lambda}
                    				\lambda^{k} = y^{k} + \rho_{k}\left(t_{k} - 1\right)\left(\AAA x^{k} -b\right).
                    			\end{equation}    		
                    		\item[3.2.] Update the sequence $\left\{ \left(x^{k} , z^{k} , y^{k}\right) \right\}_{k \in \nn}$ by
                       		\begin{align}
                           		z^{k + 1} & \in \SSS_{t_{k}}\left(z^{k} , \lambda^{k}\right), \label{FGL:PriStep} \\
                           		y^{k + 1} & = y^{k} + \mu\rho_{k}\left(\AAA z^{k + 1} - b\right), \label{FGL:MultiStep} \\
                           		x^{k + 1} & = \left(1 - t_{k}^{-1}\right)x^{k} + t_{k}^{-1}z^{k + 1}. \label{FGL:AccStep}
                       		\end{align}
 						\item[3.3.] Update the sequence $\seq{t}{k}$ by solving the equation $t_{k + 1}^{p} - t_{k}^{p} = t_{k + 1}^{p - 1}$, \ie
                       		\begin{equation} \label{FGL:Tk}
                        			t_{k + 1} = \left\{\begin{matrix}
                        			t_{k}  + 1, & \hspace{-0.6in} p = 1 \,\, \text{(convex case)}, \vspace{0.1in} \\
                        			(1 + \sqrt{1 + 4t_{k}^{2}})/2, & p =2 \,\, \text{(strongly convex case)}.
                        			\end{matrix} \right.
							\end{equation}
                    		\end{enumerate}
			\end{enumerate}\vspace{-0.2in}}}
	\end{center}
\vspace{0.2in}

	\begin{remark}[Few comments on the algorithmic framework FLAG]
		\begin{itemize}
			\item[$\rm{(i)}$] Borrowing ideas from well-known accelerated first order minimization methods (see, for instance, \cite{N83, AT2006, BT2009}), here the sequence $\seq{t}{k}$ plays a key role in accelerating the input map, \ie the \nice primal algorithmic map $\SSS_{t}\left(\cdot\right)$. To this end, the augmented parameter $\rho_{k}$ and the prox parameter $\tau_{k}$ are determined and chosen through the definition of the recursion which defines the sequence $\seq{t}{k}$ (\cf \eqref{FGL:Tk}).
			\item[$\rm{(ii)}$] Clearly, setting $t_{k} \equiv 1$, for all $k \in \nn$, in FLAG, implies $\rho_{k} \equiv \rho$, $\lambda^{k} \equiv y^{k}$, and $x^{k} \equiv z^{k}$, thus recovering the classical basic Lagrangian-based method discussed above.
			\item[$\rm{(iii)}$] A main new feature in our algorithmic framework FLAG is the introduction of the fundamental auxiliary sequence $\Seq{\lambda}{k}$, which will enable us to derive new faster rate of convergence results in terms of the {\em original} sequence produced by FLAG. As we shall see, in Section \ref{SSec:Ergodic}, when $\lambda^{k}$ coincides with $y^{k}$, only ergodic type rates (classical and fast) will be obtained.
		\end{itemize}
	\end{remark}
	In FLAG, once the \nice algorithmic map $\SSS_{t}\left(\cdot\right)$ has been chosen, \ie a parameter $\delta \in \left(0 , 1\right]$ and matrices $P , Q \in \bbS_{+}^{n}$, have been determined, Definition \ref{D:AlgoMapNice} translates in terms of the iterative sequence $\left\{ \left(x^{k} , z^{k} , y^{k}\right) \right\}_{k \in \nn}$ produced by FLAG as follows: for each $k \geq 0$ and $\xi \in \FFF$,
	\begin{align}
		\Lag_{\rho_{k}}\left(z^{k + 1} , \lambda^{k}\right) - \Lag_{\rho_{k}}\left(\xi , \lambda^{k}\right) & \leq \tau_{k}\Delta_{P}\left(\xi , z^{k} , z^{k + 1}\right) - \frac{\tau_{k}}{2}\norm{z^{k + 1} - z^{k}}_{Q}^{2}  - \frac{\sigma}{2}\norm{\xi- z^{k + 1}}^{2} \nonumber \\
		& - \frac{\delta\rho_{k}}{2}\norm{\AAA z^{k + 1} - b}^{2}, \label{AlgoMapNiceProperty}
	\end{align}		
	where $\tau_{k} = 1$ and $\rho_{k} = \rho$ (when $\sigma = 0$) or $\tau_{k} = t_{k}$ and $\rho_{k} = \rho t_{k}$ (when $\sigma > 0$).
\medskip

	To conclude this part, and before we prove that the property of being a \nice primal algorithmic map (as defined in Definition \ref{D:AlgoMapNice}) is \textit{all we need} to guarantee rate of convergence results (classical and fast), we would like to note that there is no need to enter into the mechanics of the proofs to be presented next in order to derive theoretical guarantees of newly developed/analyzed Lagrangian-based methods, but just proving that the primal algorithmic map at hand is \nice\hspace{-0.05in}! This will be extensively illustrated in Section \ref{Sec:Examples} whereby, after providing a {\em recipe} for obtaining rate of a given method, we will prove that most iconic Lagrangian-based schemes admit a \nice primal map, and hence through FLAG will generate new faster schemes.

\section{Convergence Analysis of FLAG} \label{Sec:AnalysisGL}
	The section is divided into three parts, where we start with two key lemmas, which will serve us in deriving the main {\em non-ergodic} rate of convergence results in Section \ref{SSec:NonErgodic}, followed by the ergodic type results in Section \ref{SSec:Ergodic}.
	
\subsection{Main Pillars of the Analysis} \label{SSec:Pillar}
	Before proceeding, we recall the following fundamental {\em Pythagoras three-points identity} that is frequently used below. Given any matrix $P \in \bbS_{+}^{n}$, we have
	\begin{equation} \label{Pyth}
		2\act{u - w , P\left(w - v\right)} =  \norm{u - v}_{P}^{2}- \norm{u - w}_{P}^{2} - \norm{v - w}_{P}^{2}, \quad \forall \,\, u , v , w \in \real^{n}.
	\end{equation}	
	We also recall the main parameter sequences introduced in FLAG and which, for convenience, will be used often and systematically as follows:
   \begin{equation*}
    		\rho_{k} = \rho t_{k}^{p - 1} ,\; (\rho>0)\; \quad \text{and} \quad t_{k}^{p} - t_{k - 1}^{p} = t_{k}^{p - 1},\; (t_{-1} \equiv 0 \;\text{and}\; t_0 = 1),
	\end{equation*}	
	where $p = 1$ (if $\sigma = 0$) or $p = 2$ (if $\sigma > 0$), and $\tau_{k} = 1$ if $\sigma = 0$, or $\tau_{k} = t_{k}$ if $\sigma > 0$. We will also often use the following expression. For any $k\geq 0$ and any $\xi \in {\cal F}$ we define,
	\begin{equation}\label{laggap}
		s_{k} := \Lag_{\rho t_{k - 1}^{p}}\left(x^{k} , \eta\right) - \Lag_{\rho t_{k - 1}^{p}}\left(\xi , \eta\right) \equiv \Lag_{\rho t_{k - 1}^{p}}\left(x^{k} , \eta\right) - \Psi(\xi).
	 \end{equation}
\smallskip

	\noindent We start with the first useful technical lemma.
	\begin{lemma} \label{L:FGLIter}
		Let $\left\{ \left(x^{k} , z^{k} , y^{k}\right) \right\}_{k \in \nn}$ be a sequence generated by FLAG. Then, for any $\xi \in \FFF$, $\eta \in \real^{m}$ and $k \geq 0$, we have
		\begin{align}
			\Lag_{\rho_{k}}\left(z^{k + 1} , \eta\right) - \Lag_{\rho_{k}}\left(\xi , \eta\right) & \leq \tau_{k}\Delta_{P}\left(\xi , z^{k} , z^{k + 1}\right) - \frac{\sigma}{2}\norm{\xi - z^{k + 1}}^{2} + \frac{1}{\mu\rho_{k}}\Delta\left(\eta , y^{k} , y^{k + 1}\right) \nonumber \\
			& - \rho t_{k - 1}^{p}\act{\AAA x^{k} - b , \AAA z^{k + 1} - b}. \label{L:FGLIter:0}
		\end{align}		
	\end{lemma}
	\begin{proof}
		Since $\left\{ \left(x^{k} , z^{k} , y^{k}\right) \right\}_{k \in \nn}$ is generated by FLAG, we obtain from \eqref{AlgoMapNiceProperty} after omitting the non-negative term $\left(\tau_{k}/2\right)\norm{z^{k + 1} - z^{k}}_{Q}^{2}$, that
		\begin{equation} \label{L:FGLIter:1}
			\Lag_{\rho_{k}}\left(z^{k + 1} , \lambda^{k}\right) - \Lag_{\rho_{k}}\left(\xi , \lambda^{k}\right) \leq \tau_{k}\Delta_{P}\left(\xi , z^{k} , z^{k + 1}\right) - \frac{\sigma}{2}\norm{\xi - z^{k + 1}}^{2} - \frac{\delta\rho_{k}}{2}\norm{\AAA z^{k + 1} - b}^{2}.
		\end{equation}
		From the multiplier update \eqref{FGL:MultiStep} and the three-points identity \eqref{Pyth}, we obtain, for all $\eta \in \real^{m}$,
		\begin{align}
			\act{\eta - y^{k} , \AAA z^{k + 1} - b}  = \frac{1}{\mu\rho_{k}}\act{\eta - y^{k} , y^{k + 1} - y^{k}} & = \frac{1}{\mu\rho_{k}}\Delta\left(\eta , y^{k} , y^{k + 1}\right) + \frac{1}{2\mu\rho_{k}}\norm{y^{k + 1} - y^{k}}^{2} \nonumber \\
			& = \frac{1}{\mu\rho_{k}}\Delta\left(\eta , y^{k} , y^{k + 1}\right) + \frac{\mu\rho_{k}}{2}\norm{\AAA z^{k + 1} - b}^{2}. \label{L:FGLIter:2}
		\end{align}
		Using the update rule of the sequence $\seq{t}{k}$ (\cf \eqref{FGL:Tk}) we have that $\rho_{k}\left(t_{k} - 1\right) =\rho t_{k - 1}^{p} $ and thus
		\begin{equation*}
			\lambda^{k} = y^{k} + \rho_{k}\left(t_{k} - 1\right)\left(\AAA x^{k} -b\right) = y^{k} + \rho t_{k - 1}^{p}\left(\AAA x^{k} -b\right).
		\end{equation*}
		Using this together with \eqref{L:FGLIter:2} yields, for all $\eta \in \real^{m}$, that
		\begin{align}		
			\act{\eta - \lambda^{k} , \AAA z^{k + 1} - b} & = \act{\eta - y^{k} , \AAA z^{k + 1} - b} - \rho t_{k - 1}^{p}\act{\AAA x^{k} - b , \AAA z^{k + 1} - b} \nonumber \\
			& = \frac{1}{\mu\rho_{k}}\Delta\left(\eta , y^{k} , y^{k + 1}\right) + \frac{\mu\rho_{k}}{2}\norm{\AAA z^{k + 1} - b}^{2} - \rho t_{k - 1}^{p}\act{\AAA x^{k} - b , \AAA z^{k + 1} - b}. \label{L:FGLIter:3}
		\end{align}
		Adding \eqref{L:FGLIter:1} to \eqref{L:FGLIter:3} yields (recall that $\mu \leq \delta$) the desired result.
	\end{proof}	
	From this result, we will get the (classical and fast) ergodic rate of convergence results to be presented in Section \ref{SSec:Ergodic} (see Corollaries \ref{C:FGLrate} and \ref{C:GLrate}). However, in order to obtain the {\em non-ergodic} results in Section \ref{SSec:NonErgodic}, the following result is the main step towards the derivation of the promised (classical and fast) rate of convergence results.
	\begin{lemma} \label{L:FGLMain}
		Let $\left\{ \left(x^{k} , z^{k} , y^{k}\right) \right\}_{k \in \nn}$ be a sequence generated by FLAG. Then, for any $\xi \in \FFF$, $\eta \in \real^{m}$ and $k \geq 0$, we have
		\begin{equation*}
			t_{k}^{p}s_{k + 1} - t_{k - 1}^{p}s_{k} \leq \frac{\tau_{k}\rho_{k}}{\rho}\Delta_{P}\left(\xi , z^{k} , z^{k + 1}\right) - \frac{\rho_{k}\sigma}{2\rho}\norm{\xi - z^{k + 1}}^{2} + \frac{1}{\mu\rho}\Delta\left(\eta , y^{k} , y^{k + 1}\right),
		\end{equation*}
		where $s_{k} = \Lag_{\rho t_{k - 1}^{p}}\left(x^{k} , \eta\right) - \Lag_{\rho t_{k - 1}^{p}}\left(\xi , \eta\right)$ (\cf \ref{laggap}).
	\end{lemma}	
	\begin{proof}
		Since $x^{k + 1} = \left(1 - t_{k}^{-1}\right)x^{k} + t_{k}^{-1}z^{k + 1}$ (\cf \eqref{FGL:AccStep}) we obtain from the convexity of $\Psi\left(\cdot\right)$ that
		\begin{equation*}
			\Psi\left(x^{k + 1}\right) \leq \left(1 - t_{k}^{-1}\right)\Psi\left(x^{k}\right) + t_{k}^{-1}\Psi\left(z^{k + 1}\right).
		\end{equation*}
		Therefore, multiplying both sides by $t_{k}^{p}$ (recalling that $t_{k}^{p} - t_{k}^{p - 1} = t_{k - 1}^{p}$), we obtain
		\begin{equation*}
			t_{k}^{p}\left(\Psi\left(x^{k + 1}\right) - \Psi\left(\xi\right)\right) - t_{k - 1}^{p}\left(\Psi\left(x^{k}\right) - \Psi\left(\xi\right)\right) \leq t_{k}^{p - 1}\left(\Psi\left(z^{k + 1}\right) - \Psi\left(\xi\right)\right).
		\end{equation*}
		In addition, using again \eqref{FGL:AccStep}, yields that
		\begin{equation*}
			t_{k}^{p}\act{\eta , \AAA x^{k + 1} - b} - t_{k - 1}^{p}\act{\eta , \AAA x^{k} - b} = t_{k}^{p - 1}\act{\eta , \AAA z^{k + 1} - b}.
		\end{equation*}			
		Combining these two facts yields (using the definition of the Lagrangian $\Lag$)
		\begin{equation} \label{L:FGLMain:1}
			t_{k}^{p}\left(\Lag\left(x^{k + 1} , \eta\right) - \Lag\left(\xi , \eta\right)\right) - t_{k - 1}^{p}\left(\Lag\left(x^{k} , \eta\right) - \Lag\left(\xi , \eta\right)\right) \leq t_{k}^{p - 1}\left(\Lag\left(z^{k + 1} , \eta\right) - \Lag\left(\xi , \eta\right)\right).
		\end{equation}			
		Now, we will again use the relation $x^{k + 1} = \left(1 - t_{k}^{-1}\right)x^{k} + t_{k}^{-1}z^{k + 1}$ to obtain
		\begin{equation*}
			\norm{\AAA x^{k + 1} - b}^{2}  = \left(1 - t_{k}^{-1}\right)^{2}\norm{\AAA x^{k} - b}^{2} + t_{k}^{-2}\norm{\AAA z^{k + 1} - b}^{2} + 2\left(1 - t_{k}^{-1}\right)t_{k}^{-1}\act{\AAA x^{k} - b , \AAA z^{k + 1} - b}.
		\end{equation*}		
		Multiplying both sides of the above equality by $\rho t_{k}^{2p}/2$ (recalling $\rho_{k} = \rho t_{k}^{p - 1}$ and $t_{k}^{p} - t_{k}^{p - 1} = t_{k - 1}^{p}$) yields
		\begin{align}
			\frac{\rho t_{k}^{2p}}{2}\norm{\AAA x^{k + 1} - b}^{2} - \frac{\rho t_{k - 1}^{2p}}{2}\norm{\AAA x^{k} - b}^{2} = \frac{\rho_{k}t_{k}^{p - 1}}{2}\norm{\AAA z^{k + 1} - b}^{2} + \rho_{k}t_{k - 1}^{p}\act{\AAA x^{k} - b , \AAA z^{k + 1} - b}. \label{L:FGLMain:2}
		\end{align}
		Therefore, adding \eqref{L:FGLMain:1} to \eqref{L:FGLMain:2} yields (recall that $\norm{\AAA\xi - b}^{2} = 0$) with $s_{k} = \Lag_{\rho t_{k - 1}^{p}}\left(x^{k} , \eta\right) - \Lag_{\rho t_{k - 1}^{p}}\left(\xi , \eta\right)$,
		\begin{equation*}
			t_{k}^{p}s_{k + 1} - t_{k - 1}^{p}s_{k} \leq t_{k}^{p - 1}\left(\Lag_{\rho_{k}}\left(z^{k + 1} , \eta\right) - \Lag_{\rho_{k}}\left(\xi , \eta\right)\right) + \rho_{k}t_{k - 1}^{p}\act{\AAA x^{k} - b , \AAA z^{k + 1} - b}.
		\end{equation*}			
		From Lemma \ref{L:FGLIter}, after we multiplied both sides of \eqref{L:FGLIter:0} by $t_{k}^{p - 1}$ (recall that $\rho_{k} = \rho t_{k}^{p - 1}$), we obtain that
		\begin{align*}
			t_{k}^{p - 1}\left(\Lag_{\rho_{k}}\left(z^{k + 1} , \eta\right) - \Lag_{\rho_{k}}\left(\xi , \eta\right)\right) + \rho_{k}t_{k - 1}^{p}\act{\AAA x^{k} - b , \AAA z^{k + 1} - b} & \leq \frac{\tau_{k}\rho_{k}}{\rho}\Delta_{P}\left(\xi , z^{k} , z^{k + 1}\right) - \frac{\rho_{k}\sigma}{2\rho}\norm{\xi - z^{k + 1}}^{2} \\
			& + \frac{1}{\mu\rho}\Delta\left(\eta , y^{k} , y^{k + 1}\right).
		\end{align*}		
		By combining the last two inequalities, we thus get
		\begin{equation*}
			t_{k}^{p}s_{k + 1} - t_{k - 1}^{p}s_{k} \leq \frac{\tau_{k}\rho_{k}}{\rho}\Delta_{P}\left(\xi , z^{k} , z^{k + 1}\right) - \frac{\rho_{k}\sigma}{2\rho}\norm{\xi - z^{k + 1}}^{2} + \frac{1}{\mu\rho}\Delta\left(\eta , y^{k} , y^{k + 1}\right),
		\end{equation*}				
		which proves the desired result.
	\end{proof}		
	We conclude this section with the following simple result which summarizes some useful properties of the sequence $\seq{t}{k}$ (\cf \eqref{FGL:Tk}) that will be used below.
	\begin{lemma} \label{L:Tk}
		Let $\seq{t}{k}$ be the sequence of positive numbers generated from $t_{0} = 1$ via the equation $t_{k + 1}^{p} - t_{k}^{p} = t_{k + 1}^{p - 1}$, with $p = 1$ or $p = 2$. Then, for any $k \geq 0$, we have
	\begin{itemize}
		\item[$\rm{(i)}$] For $p = 1$, we have $t_{k} = k + 1$.
		\item[$\rm{(ii)}$] For $p = 2$, we have $t_{k} \geq \left(k + 1\right)/2$. Moreover, one has $t_{k + 1}^{2} \leq t_{k}^{2} + 2t_{k}$.
	\end{itemize}				
	\end{lemma}
	\begin{proof}
		The first item is immediate as well as the first part of the second item (see, for instance, \cite[Lemma 4.3]{BT2009}). Moreover, note that by the given recursion $t_{k + 1}^{2} - t_{k}^{2} = t_{k + 1}$, one also has $t_{k + 1} = t_{k} + t_{k + 1}/\left(t_{k} + t_{k + 1}\right)$, and hence it follows that $t_{k + 1}^{2} = t_{k}^{2} + t_{k + 1} \leq t_{k}^{2} + 2t_{k}$, since $t_{k} \geq 1$.
	\end{proof}
	In the rest of this section we aim at proving rate of convergence results of FLAG. We will focus, in this paper, on iteration complexity of Lagrangian-based methods, using the following two classical measures at a given iterate $k \in \nn$:
	\begin{itemize}
		\item[$\rm{(i)}$] \textit{Function values gap} in terms of $\Psi\left(x^{k}\right) - \Psi\left(x^{\ast}\right)$.
		\item[$\rm{(ii)}$] \textit{Feasibility violation} of the constraints of problem (P) in terms of $\norm{\AAA x^{k} - b}$.
	\end{itemize}		
	 When discussing the measures mentioned above for the iterates produced by FLAG, we will also distinguish our results between rates expressed in terms of the original produced sequence or of the ergodic sequence, as defined below in subsection \ref{SSec:Ergodic}.

\subsection{Non-Ergodic Rate of Convergence} \label{SSec:NonErgodic}
	Throughout the rest of this section, recalling the fact that the set of optimal dual solutions is compact (see Section \ref{Sec:Form}), we take $c > 0$ to be a constant for which $c \geq 2\norm{y^{\ast}}$, where $y^{\ast}$ is an optimal solution of the dual problem.
\medskip

	We are now ready to prove our main result in the strongly convex setting (\ie $p = 2$) on the sequence itself.
	\begin{theorem}[A fast non-ergodic function values and feasibility violation rates] \label{T:FGLrate}
		Let $\left\{ \left(x^{k} , z^{k} , y^{k}\right) \right\}_{k \in \nn}$ be a sequence generated by FLAG. Suppose that $\sigma > 0$ and $0 \preceq P \preceq \left(\sigma/2\right)I_{n}$. Then, for any optimal solution $x^{\ast}$ of problem (P), we have
   		\begin{align}
			\Psi\left(x^{N}\right) - \Psi\left(x^{\ast}\right) & \leq \frac{B_{\rho , c}\left(x^{\ast}\right)}{2N^{2}}, \label{T:FGLrate:1} \\
       		\norm{\AAA x^{N} - b} & \leq \frac{B_{\rho , c}\left(x^{\ast}\right)}{cN^{2}}, \label{T:FGLrate:2}
		\end{align}
		where $B_{\rho , c}\left(x^{\ast}\right) := 4\left(\norm{x^{\ast} - z^{0}}_{P}^{2} + \frac{1}{\mu\rho}\left(\norm{y^{0}} + c\right)^{2}\right)$.
	\end{theorem}
	\begin{proof}
		From Lemma \ref{L:FGLMain}, after we substitute $\xi = x^{\ast}$ and $p = 2$ (recall that $\rho_{k} = \rho t_{k}$ and $\tau_{k} = t_{k}$), we obtain
		\begin{equation*}
			t_{k}^{2}s_{k + 1} - t_{k - 1}^{2}s_{k} \leq t_{k}^{2}\Delta_{P}\left(x^{\ast} , z^{k} , z^{k + 1}\right) - \frac{t_{k}\sigma}{2}\norm{x^{\ast} - z^{k + 1}}^{2} + \frac{1}{\mu\rho}\Delta\left(\eta , y^{k} , y^{k + 1}\right).
		\end{equation*}
		Using the definition of $\Delta_{P}$ and $\Delta$ (see \eqref{D:Delta}) and the fact that $P \preceq \left(\sigma/2\right)I_{n}$ we have that
		\begin{align*}
			t_{k}^{2}s_{k + 1} - t_{k - 1}^{2}s_{k} & \leq \frac{t_{k}^{2}}{2}\left(\norm{x^{\ast} - z^{k}}_{P}^{2} - \norm{x^{\ast} - z^{k + 1}}_{P}^{2}\right) - t_{k}\norm{x^{\ast} - z^{k + 1}}_{P}^{2} + \frac{1}{2\mu\rho}\left(\norm{\eta - y^{k}}^{2} -  \norm{\eta - y^{k + 1}}^{2}\right) \nonumber \\
			& \leq \frac{1}{2}\left(t_{k}^{2}\norm{x^{\ast} - z^{k}}_{P}^{2} - t_{k + 1}^{2}\norm{x^{\ast} - z^{k + 1}}_{P}^{2}\right) + \frac{1}{2\mu\rho}\left(\norm{\eta - y^{k}}^{2} - \norm{\eta - y^{k + 1}}^{2}\right),
		\end{align*}
		where the last inequality follows from Lemma \ref{L:Tk}(ii). Summing this inequality for all $k = 0 , 1 , \ldots , N - 1$ (recall that $t_{-1} = 0$ and $t_{0} = 1$), it follows (where we use  the definition of $s_{N}$, see \eqref{laggap}, for the first inequality below)
		\begin{equation*}
	t_{N - 1}^{2}\left(\Psi\left(x^{N}\right) - \Psi\left(x^{\ast}\right) + \act{\eta , \AAA x^{N} - b}\right) \leq	t_{N - 1}^{2}s_{N} - t_{-1}^{2}s_{0} \leq \frac{1}{2}\norm{x^{\ast} - z^{0}}_{P}^{2} + \frac{1}{2\mu\rho}\norm{\eta - y^{0}}^{2}.
		\end{equation*}
		Therefore, by taking the maximum of both sides over $\norm{\eta} \leq c$, proves that (recall that $t_{N - 1}^{2} \geq N^{2}/4$ using Lemma \ref{L:Tk}(ii))
		\begin{equation*}
			\Psi\left(x^{N}\right) - \Psi\left(x^{\ast}\right) + c\norm{\AAA x^{N} - b} \leq \frac{B_{\rho , c}\left(x^{\ast}\right)}{2N^{2}} := \beta,
		\end{equation*}
		from which the estimate \eqref{T:FGLrate:1} immediately follows. Moreover, since $\left(x^{\ast} , y^{\ast}\right)$ is a saddle point of problem (P), with $c \geq 2\norm{y^{\ast}}$, using the above inequality it follows that
		\begin{equation*}
			c\norm{\AAA x^{N} - b} \leq \Psi\left(x^{\ast}\right) - \Psi\left(x^{N}\right) + \beta \leq \act{y^{\ast} , \AAA x^{N} - b} + \beta \leq \frac{c}{2}\norm{\AAA x^{N} - b} + \beta,
		\end{equation*}
		and the desired estimate \eqref{T:FGLrate:2} is proved.
	\end{proof}		
	Our main result in the convex setting (\ie $\sigma = 0$ and $p = 1$) on the sequence itself is recorded next.
	\begin{theorem}[A non-ergodic function values and feasibility violation rates] \label{T:GLrate}
		Let $\left\{ \left(x^{k} , z^{k} , y^{k}\right) \right\}_{k \in \nn}$ be a sequence generated by FLAG and suppose that $\sigma = 0$. Then, for any optimal solution $x^{\ast}$ of problem (P), we have
   		\begin{align}
			\Psi\left(x^{N}\right) - \Psi\left(x^{\ast}\right) & \leq \frac{B_{\rho , c}\left(x^{\ast}\right)}{2N}, \label{T:GLrate:1} \\
       		\norm{\AAA x^{N} - b} & \leq \frac{B_{\rho , c}\left(x^{\ast}\right)}{cN}, \label{T:GLrate:2}
		\end{align}
		where $B_{\rho , c}\left(x^{\ast}\right) := 2\left(\norm{x^{\ast} - z^{0}}_{P}^{2} + \frac{1}{\mu\rho}\left(\norm{y^{0}} + c\right)^{2}\right)$.
	\end{theorem}
	\begin{proof}
		From Lemma \ref{L:FGLMain}, after we substitute $\xi = x^{\ast}$, $\sigma = 0$ and $p = 1$ (recall that $\rho_{k} = \rho$ and $\tau_{k}=1 $), we obtain
		\begin{equation*}
			t_{k}s_{k + 1} - t_{k - 1}s_{k} \leq \Delta_{P}\left(x^{\ast} , z^{k} , z^{k + 1}\right) + \frac{1}{\mu\rho}\Delta\left(\eta , y^{k} , y^{k + 1}\right).
		\end{equation*}
		Using the definition of $\Delta_{P}$ and $\Delta$ (see \eqref{D:Delta}) we have that
		\begin{align*}
			t_{k}s_{k + 1} - t_{k - 1}s_{k} & \leq \frac{1}{2}\left(\norm{x^{\ast} - z^{k}}_{P}^{2} - \norm{x^{\ast} - z^{k + 1}}_{P}^{2}\right) + \frac{1}{2\mu\rho}\left(\norm{\eta - y^{k}}^{2} -  \norm{\eta - y^{k + 1}}^{2}\right).
		%	& \leq \frac{1}{2}\left(\frac{1}{t_{k}}\norm{x^{\ast} - z^{k}}_{P}^{2} - \frac{1}{t_{k + 1}}\norm{x^{\ast} - z^{k + 1}}_{P}^{2}\right) + \frac{1}{2\mu\rho}\left(\norm{\eta - y^{k}}^{2} -  \norm{\eta - y^{k + 1}}^{2}\right),
		\end{align*}
		%where the second inequality follows from the fact that $-t_{k}^{-1} \leq -t_{k + 1}^{-1}$.
 Summing this inequality for all $k = 0 , 1 , \ldots , N - 1$ (recall that $t_{-1} = 0$ and $t_{0} =1$) it follows
		\begin{equation*}
			t_{N - 1}\left(\Psi\left(x^{N}\right) - \Psi\left(x^{\ast}\right) + \act{\eta , \AAA x^{N} - b}\right) \leq	t_{N - 1}s_{N} - t_{-1}s_{0} = t_{N - 1}s_{N}\leq \frac{1}{2}\norm{x^{\ast} - z^{0}}_{P}^{2} + \frac{1}{2\mu\rho}\norm{\eta - y^{0}}^{2}.
		\end{equation*}
		Therefore, by taking the maximum of both sides over $\norm{\eta} \leq c$, proves that (recall that $t_{N - 1} = N$ using Lemma \ref{L:Tk}(i))
		\begin{equation*}
			\Psi\left(x^{N}\right) - \Psi\left(x^{\ast}\right) + c\norm{\AAA x^{N} - b} \leq \frac{B_{\rho , c}\left(x^{\ast}\right)}{2N},
		\end{equation*}
		and the desired results follow exactly as done at the end of the proof of Theorem \ref{T:FGLrate}.
	\end{proof}	
	
\subsection{Classical Ergodic Rate of Convergence} \label{SSec:Ergodic}
	Although deriving ergodic type results was not our primary goal, to conclude this section we further illustrate the versatility of our framework by providing the  fast and classical  {\em ergodic} rate of convergence results. Before doing so, it should be noted that in the setting of producing such weaker results of ergodic type, two main features of our main framework FLAG are simplified as follows:
%\textcolor{red}{CHECK AGAIN!}
	\begin{itemize}
		\item[(i)] The auxiliary sequence $\Seq{\lambda}{k}$ is simply replaced by the multiplier sequence $\Seq{y}{k}$, i.e.,  the sequence  $\Seq{x}{k}$
		plays no role. Thus, $\Seq{z}{k}$ is the only primal sequence, so that the main iteration \eqref{FGL:PriStep} reads now $z^{k + 1} \in \SSS_{t_{k}}\left(z^{k} , y^{k}\right)$.
		\item[(ii)] Only the parameter $\rho_{k}$ and $\tau_{k}$ are still used to include both the fast and classical results.
	\end{itemize}
	According to the above, since we replace the auxiliary multiplier $\lambda^{k}$ with $y^{k}$, it is easy to verify from the proof of Lemma \ref{L:FGLIter} (by combining \eqref{L:FGLIter:1} and \eqref{L:FGLIter:2} there) that in this case we have
		\begin{equation} \label{R:ErgoGL:1}
			\Lag_{\gamma_{k}}\left(z^{k + 1} , \eta\right) - \Lag_{\gamma_{k}}\left(\xi , \eta\right) \leq \tau_{k}\Delta_{P}\left(\xi , z^{k} , z^{k + 1}\right) - \frac{\sigma}{2}\norm{\xi - z^{k + 1}}^{2} + \frac{1}{\mu\rho_{k}}\Delta\left(\eta , y^{k} , y^{k + 1}\right),
		\end{equation}				
		where $\gamma_{k} := \left(1 + \delta - \mu\right)\rho_{k} \geq 0$. Therefore, in this case we can consider scaling parameters $\mu$ in the larger interval $\left(0 , 1 + \delta\right]$, instead of $\left(0 , \delta\right]$ as stated in the original version of FLAG above.
\medskip

	We begin with the result in the strongly convex setting (\ie $p = 2$).
	\begin{corollary}[A fast ergodic function values and feasibility violation rates] \label{C:FGLrate}
		Let $\left\{ \left(z^{k} , y^{k}\right) \right\}_{k \in \nn}$ be a sequence generated by FLAG. Suppose that $\sigma >0$ and $0 \preceq P \preceq \left(\sigma/2\right)I_{n}$. Then, for any optimal solution $z^{\ast}$ of problem (P), the following holds for the ergodic sequence ${\bar z}^{N} = t_{N -1}^{-2}\sum_{k = 0}^{N - 1} t_{k}z^{k + 1}$
   		\begin{align}
			\Psi\left({\bar z}^{N}\right) - \Psi\left(z^{\ast}\right) & \leq \frac{B_{\rho , c}\left(z^{\ast}\right)}{2N^{2}}, \label{C:FGLrate:1} \\
       		\norm{\AAA {\bar z}^{N} - b} & \leq \frac{B_{\rho , c}\left(z^{\ast}\right)}{cN^{2}}, \label{C:FGLrate:2}
		\end{align}
		where $B_{\rho , c}\left(z^{\ast}\right) := 4\left(\norm{z^{\ast} - z^{0}}_{P}^{2} + \frac{1}{\mu\rho}\left(\norm{y^{0}} + c\right)^{2}\right)$.
	\end{corollary}
	\begin{proof}
		Substituting in \eqref{R:ErgoGL:1} $\xi = z^{\ast}$ and multiplying both sides by $t_{k}$ yields (recall that $\rho_{k} = \rho t_{k}$ and $\tau_{k} = t_{k}$)
		\begin{equation*}
			t_{k}\left(\Psi\left(z^{k + 1}\right) - \Psi\left(z^{\ast}\right) + \act{\eta , \AAA z^{k + 1} - b}\right) \leq t_{k}^{2}\Delta_{P}\left(z^{\ast} , z^{k} , z^{k + 1}\right) - \frac{t_{k}\sigma}{2}\norm{z^{\ast} - z^{k + 1}}^{2} + \frac{1}{\mu\rho}\Delta\left(\eta , y^{k} , y^{k + 1}\right).
		\end{equation*}
		where we have omitted the non-negative augmented term, which is present only in $\Lag_{\gamma_{k}}\left(z^{k + 1} , \eta\right)$. Using the definition of $\Delta_{P}$ and $\Delta$ (see \eqref{D:Delta}) and the fact that $P \preceq \left(\sigma/2\right)I_{n}$ we thus obtain
		\begin{align*}
			t_{k}\left(\Psi\left(z^{k + 1}\right) - \Psi\left(z^{\ast}\right) + \act{\eta , \AAA z^{k + 1} - b}\right) & \leq \frac{t_{k}^{2}}{2}\left(\norm{z^{\ast} - z^{k}}_{P}^{2} - \norm{z^{\ast} - z^{k + 1}}_{P}^{2}\right) - t_{k}\norm{z^{\ast} - z^{k + 1}}_{P}^{2} \nonumber \\
			& + \frac{1}{\mu\rho}\Delta\left(\eta , y^{k} , y^{k + 1}\right) \nonumber \\
			& \leq \frac{1}{2}\left(t_{k}^{2}\norm{z^{\ast} - z^{k}}_{P}^{2} - t_{k + 1}^{2}\norm{z^{\ast} - z^{k + 1}}_{P}^{2}\right) \nonumber \\
			& + \frac{1}{2\mu\rho}\left(\norm{\eta - y^{k}}^{2} - \norm{\eta - y^{k + 1}}^{2}\right),
		\end{align*}
		where the last inequality follows from Lemma \ref{L:Tk}(ii). Summing this inequality for all $k = 0 , 1 , \ldots , N - 1$ it follows that (recall that $t_{0} = 1$)
		\begin{equation*}
			\sum_{k = 0}^{N - 1} t_{k}\left(\Psi\left(z^{k + 1}\right) - \Psi\left(z^\ast \right) + \act{\eta , \AAA z^{k + 1} - b}\right) \leq \frac{1}{2}\norm{z^{\ast} - z^{0}}_{P}^{2} + \frac{1}{2\mu\rho}\norm{\eta - y^{0}}^{2}.
		\end{equation*}
		Using the relation $t_{k} = t_{k}^{2} - t_{k - 1}^{2}$ and the definition of ${\bar z}^{N}$, it easily follows (recall that $t_{-1} = 0$)
		\begin{equation*}
			\sum_{k = 0}^{N - 1} t_{k} = t_{N - 1}^{2} \quad \text{and} \quad \sum_{k = 0}^{N - 1} t_{k} \act{\eta , \AAA z^{k + 1} - b}=t_{N - 1}^{2} \act{\eta , \AAA{\bar z}^{N}-b}.
		\end{equation*}
		Therefore, thanks to Jensen's inequality for the convex function $\Psi$, and these relations, it follows from the  inequality above that
		\begin{equation*}
			t_{N - 1}^{2}\left(\Psi\left({\bar z}^{N}\right) - \Psi\left(z^{\ast}\right) + \act{\eta , \AAA{\bar z}^{N} - b}\right) \leq \frac{1}{2}\norm{z^{\ast} - z^{0}}_{P}^{2} + \frac{1}{2\mu\rho}\norm{\eta - y^{0}}^{2}.
		\end{equation*}
		Taking the maximum of both sides over $\norm{\eta} \leq c$, proves that (recall that $t_{N - 1}^{2} \geq N^{2}/4$ using Lemma \ref{L:Tk}(ii))
		\begin{equation*}
			\Psi\left({\bar z}^{N}\right) - \Psi\left(z^{\ast}\right) + c\norm{\AAA{\bar z}^{N} - b} \leq \frac{B_{\rho , c}\left(z^{\ast}\right)}{N^{2}}.
		\end{equation*}
		The required results now follow as done at the end of the proof of Theorem \ref{T:FGLrate}.
	\end{proof}		
	The classical ergodic rate of convergence result in the convex setting (\ie $\sigma = 0$ and $p = 1$) is recorded next.
	\begin{corollary}[An ergodic function values and feasibility violation rates] \label{C:GLrate}
		Let $\left\{ \left(z^{k} , y^{k}\right) \right\}_{k \in \nn}$ be a sequence generated by FLAG with $\sigma = 0$ and $t_{k} = 1$ for all $k \in \nn$. Then, for any optimal solution $z^{\ast}$ of problem (P), the following holds for the ergodic sequence ${\bar z}^{N} = N^{-1}\sum_{k = 0}^{N - 1} z^{k + 1}$
   		\begin{align}
			\Psi\left({\bar z}^{N}\right) - \Psi\left(z^{\ast}\right) & \leq \frac{B_{\rho , c}\left(z^{\ast}\right)}{2N}, \label{CT:GLrate:1} \\
       		\norm{\AAA {\bar z}^{N} - b} & \leq \frac{B_{\rho , c}\left(z^{\ast}\right)}{cN}, \label{C:GLrate:2}
		\end{align}
		where $B_{\rho , c}\left(z^{\ast}\right) := 2\left(\norm{z^{\ast} - x^{0}}_{P}^{2} + \frac{1}{\mu\rho}\left(\norm{y^{0}} + c\right)^{2}\right)$.
	\end{corollary}
	\begin{proof}
		The result trivially follows from the proof of Corollary \ref{C:FGLrate} since $t_{k} = 1$ for all $k \in \nn$.
	\end{proof}
	As discussed in the introduction, similar ergodic $O(1/N)$ results have been obtained in many works in the literature and on many variants of Lagrangian-based methods; see, for instance, \cite{CP2011,HY2012}{\bl \cite{MS2013}},\cite{
ST2014,ST2019, LL2019,LLF2020} and references therein.	
	
\section{Applications: Nice Primal Algorithmic Maps and Their FLAG} \label{Sec:Examples}
	In this section, we will show that several well-known iconic Lagrangian-based methods {\em all admit} \nice primal algorithmic maps. Recall that the notion of \nice primal algorithmic maps deals only with the update of the primal variable(s), while the multiplier update is the same for most Lagrangian-based methods, which in FLAG, is simply given by:
	\begin{equation} \label{MultiUpdate}
		y^{+} =  y + \mu\rho_{t}\left(\AAA z^{+} - b\right),
	\end{equation}
	where $\rho_{t} = \rho$ in the convex case or $\rho_{t} = \rho t$ in the strongly convex case. Thus, given any Lagrangian-based method, we only need to consider the corresponding primal algorithmic map $\SSS_{t}\left(\cdot\right)$ and use it in FLAG to produce a new variant of this algorithm with faster {\em non-ergodic} rate of convergence: $O(1/N^{2})$ when the objective is strongly convex and $O(1/N)$ when the objective is convex, as proved in Theorems \ref{T:FGLrate} and \ref{T:GLrate}, respectively.
\medskip

	It is well-known that a major drawback in Lagrangian-based methods is the presence of the quadratic augmented penalty term, which often renders a primal minimization step as a difficult one. This difficulty can be overcomed by {\em linearizing the quadratic} expression around a given current iterate. An important point to stress in this respect, is that in all of the instances below, we will show that we can always consider linearizing the quadratic penalty term, and prove that it generates a corresponding {\em \nice} primal algorithmic map $\SSS_{t}\left(\cdot\right)$.
\medskip

	 Before presenting these results, for a potential user who is interested in developing and/or analyzing a Lagrangian-based method for solving a certain convex optimization problem, we offer the following informal recipe for obtaining faster rates of the designed method.
\bigskip

	\noindent {\bf A recipe for finding the rate of convergence  of Lagrangian-based methods.}
	\begin{itemize}
		\item[$\rm{(i)}$] Formulate the problem at hand via model (P), \ie identify the relevant problem data $[\Psi, \AAA, b, \sigma]$, where the value of the modulus of convexity $\sigma$ will infer the value of $p$ ($1$ or $2$) to be used in FLAG,  and which determine the type of rate that can be achieved (classical or fast).
		\item[$\rm{(ii)}$] Define the desired  iterative step(s) of the primal algorithmic map $\SSS_{t}\left(\cdot\right)$ applied on the augmented Lagrangian $\Lag_{\rho_{t}}\left(\cdot\right)$ of model (P).
		\item[$\rm{(iii)}$] Show that the defined primal algorithmic map is \nice according to Definition \ref{D:AlgoMapNice} (\ie determine the parameter $\delta$ and the matrices $P$ and $Q$) using information on the involved objective function $\Psi\left(\cdot\right)$, the linear mapping $\AAA$ and the iterative step(s) defining $\SSS_{t}\left(\cdot\right)$.
		\item[$\rm{(iv)}$] Apply Theorem \ref{T:FGLrate} (if $p = 2$) or Theorem \ref{T:GLrate} (if $p = 1$)  to obtain a faster non-ergodic rate of convergence for the designed method.
	\end{itemize}	
	To proceed, we recall the following well-known result that will be useful below. For completeness we include its simple proof.
	\begin{lemma}[Composite proximal inequality] \label{L:ProxIneqLag}
		Let $\varphi : \real^{d} \rightarrow \erl$ be a proper, lower-semi continuous and $\sigma$-strongly convex function ($\sigma \geq 0$) and let $c : \real^{d} \rightarrow \real$ be a differentiable function. Given a matrix $W \in \bbS_{+}^{d}$, we define:
		\begin{equation*}
			w^{+} \in \argmin_{\xi \in \real^{d}} \left\{ \varphi\left(\xi\right) + c\left(\xi\right) + \frac{1}{2}\norm{\xi - w}_{W}^{2} \right\}.		
		\end{equation*}
		Then, for any $\xi \in \real^{d}$, we have
		\begin{equation*}
			\varphi\left(w^{+}\right) - \varphi\left(\xi\right) + \act{\nabla c\left(w^{+}\right) , w^{+} - \xi} \leq \frac{1}{2}\left(\norm{\xi - w}_{W}^{2} -  \norm{\xi - w^{+}}_{W}^{2} -  \norm{w^{+} - w}_{W}^{2}\right) - \frac{\sigma}{2}\norm{\xi - w^{+}}^{2}.
		\end{equation*}
	\end{lemma}
	\begin{proof}
		Writing the first-order optimality condition of the optimization problem yields, for any $\xi \in \real^{d}$
		\begin{equation*}
			\bo \in \partial\varphi\left(w^{+}\right) + \nabla c\left(w^{+}\right) + W\left(w^{+} - w\right),
		\end{equation*}
		which means that
		\begin{equation*}
			\nabla c\left(w^{+}\right) + W\left(w^{+} - w\right) \equiv-\xi^{+} \in -\partial\varphi\left(w^{+}\right).
		\end{equation*}		
		On the other hand, using the $\sigma$-strong convexity of $\varphi$, for all $\xi \in \real^{d}$, we have
		\begin{equation*}
			\varphi\left(w^{+}\right) - \varphi\left(\xi\right) \leq \act{-\xi^{+} , \xi - w^{+}} - \frac{\sigma}{2}\norm{\xi - w^{+}}^{2} = \act{\nabla c\left(w^{+}\right) + W\left(w^{+} - w\right), \xi - w^{+}} - \frac{\sigma}{2}\norm{\xi - w^{+}}^{2}.
		\end{equation*}
		By rearranging terms and using the Pythagoras identity \eqref{Pyth} we obtain the desired result.
	\end{proof}		
	
\subsection{Methods for the Basic Model (P)}
	As we discussed in Section \ref{Sec:GL}, the most basic example of a primal algorithmic map is the mapping $\SSS_{t}\left(\cdot\right)$ that performs an \textit{exact minimization} on the augmented Lagrangian itself, or a proximal variant of it.
	
\subsubsection{Example I: Augmented Lagrangian} \label{SSec:AL}
	For the classical (Proximal) Augmented Lagrangian method, the primal step of FLAG is given, in this case, by
	\begin{equation} \label{PAL:Prim}
		z^{+} = \SSS_{t}\left(z , \lambda\right) \equiv \argmin_{\xi} \left\{ \Psi\left(\xi\right) + \act{\lambda , \AAA\xi - b} + \frac{\rho_{t}}{2}\norm{\AAA\xi - b}^{2} + \frac{\tau_{t}}{2}\norm{\xi - z}_{M}^{2} \right\},
	\end{equation}		
	where $M \succeq 0$ is a weight matrix defining the additional proximal term. Therefore, if $M \neq 0$ then the algorithmic map $\SSS_{t}\left(\cdot\right)$ can be seen as a {\em proximal minimization} applied on the augmented Lagrangian $\Lag_{\rho_{t}}\left(\cdot , y\right)$, for a fixed $y \in \real^{m}$. In this case, we can show the following result, which simple proof is omitted; instead see the proof of Lemma \ref{L:NiceLAL} below (which essentially includes this lemma with an adequate choice of the involved matrices).
	\begin{lemma}[Proximal AL is nice] \label{L:NicePAL}
		Let $M \succeq 0$, the algorithmic map $\SSS_{t}\left(\cdot\right)$ defined in \eqref{PAL:Prim} is \nice with $\delta = 1$ and $P = Q = M \in \bbS_{+}^{n}$.
	\end{lemma}

\subsubsection{Example II: Proximal Linearized Augmented Lagrangian} \label{SSec:LAL}
	It should be noted that the Proximal AL method mentioned above requires to solve an optimization problem, which is usually not easy (even when $\Psi\left(\cdot\right)$ is prox-friendly) mainly because of the {\em quadratic augmented} term. One classical way to overcome this difficulty is to {\em linearize} the augmented term around the current iterate. Therefore, by choosing the algorithmic map $\SSS_{t}\left(\cdot\right)$ to be a {\em proximal gradient}  step applied on the augmented Lagrangian $\Lag_{\rho_{t}}$ yields the {\em Proximal Linearized AL Method}, see, \eg \cite{HY2012, ST2014}. In this case, we prove the following result.
	\begin{lemma}[Proximal Linearized AL is nice] \label{L:NiceLAL}
		Let $M \succ 0$, then the algorithmic map defined by
		\begin{equation*}
			z^{+} = \SSS_{t}\left(z , \lambda\right) \equiv \argmin_{\xi} \left\{ \Psi\left(\xi\right) + \act{\lambda , \AAA\xi - b} + \rho_{t}\act{\AAA z - b , \AAA\xi} + \frac{\tau_{t}}{2}\norm{\xi - z}_{M}^{2} \right\},
		\end{equation*}
		is nice with $\delta = 1$ and $P = Q =  M - \rho\AAA^{T}\AAA \succeq 0$.
	\end{lemma}
	\begin{proof}
		Applying Lemma \ref{L:ProxIneqLag} with the function $\varphi\left(\cdot\right) = \Psi\left(\cdot\right)$, $c\left(\xi\right) = \act{\lambda + \rho_{t}\left(\AAA z - b\right) , \AAA\xi}$ and $W = \tau_{t}M$ yields for any $\xi \in \real^{n}$ that
		\begin{equation} \label{L:NiceLAL:1}
			\Psi\left(z^{+}\right) - \Psi\left(\xi\right) + \act{\lambda + \rho_{t}\left(\AAA z - b\right) , \AAA z^{+} - \AAA\xi} \leq \tau_{t}\Delta_{M}\left(\xi , z , z^{+}\right) - \frac{\tau_{t}}{2}\norm{z^{+} - z}_{M}^{2} - \frac{\sigma}{2}\norm{\xi - z^{+}}^{2}.
		\end{equation}
		Taking $\xi \in \FFF$ yields that $\AAA\xi = b$, and thus we can estimate the following term as:
		\begin{align}
			\act{\AAA z - b , \AAA z^{+} - \AAA\xi} & = \act{\AAA z^{+} - b , \AAA z^{+} - b} + \act{\AAA z - \AAA z^{+} , \AAA z^{+} - \AAA\xi} \nonumber \\
			& = \norm{\AAA z^{+} - b}^{2} + \Delta_{\AAA ^{T}\AAA}\left(\xi , z , z^{+}\right) - \frac{1}{2}\norm{z^{+} - z}_{\AAA ^{T}\AAA}^{2}. \label{L:NiceLAL:2}
		\end{align}
		Using \eqref{L:NiceLAL:2} in \eqref{L:NiceLAL:1} we obtain,
		\begin{align*}
			\Psi\left(z^{+}\right) - \Psi\left(\xi\right) + \act{\lambda , \AAA z^{+} - b} + \frac{\rho_{t}}{2}\norm{\AAA z^{+} - b}^{2} & \leq
			\tau_{t}\Delta_{M}\left(\xi , z , z^{+}\right) - \frac{\tau_{t}}{2}\norm{z^{+} - z}_{M}^{2} - \frac{\sigma}{2}\norm{\xi - z^{+}}^{2} \\
			&  - \rho_t \Delta_{\AAA ^{T}\AAA}\left(\xi , z , z^{+}\right) + \frac{\rho_t}{2}\norm{z^{+} - z}_{\AAA ^{T}\AAA}^{2} - \frac{\rho_t}{2}\norm{\AAA z^{+} - b}^{2}.
		\end{align*}
		Hence, the \nice property of $\SSS_{t}\left(\cdot\right)$ follows from Definition \ref{D:AlgoMapNice}, where we also used the fact that $\rho_{t} = \rho\tau_{t}$ and the definition of the augmented Lagrangian (see \eqref{D:AugLA}), which implies that $\Lag_{\rho}\left(\xi , \lambda\right) = \Psi\left(\xi\right)$ for all $\xi \in \FFF$.
	\end{proof}		

\subsection{Methods for Block Models} \label{SSec:Block}
	The Lagrangian-based algorithms to be presented below are designed to tackle the block model as discussed in Example \ref{E:BLCM}:
	\begin{equation*}
		\min_{(u , v) \in \real^{n}} \left\{ f\left(u\right) + g\left(v\right) \, : \, Au + Bv = b \right\},
	\end{equation*}
	which fits the general model (P) where the linear mapping $\AAA$ is defined by $\AAA z = Au + Bv$ and $\Psi\left(z\right) = f\left(u\right) + g\left(v\right)$ where $x = \left(u^{T} , v^{T}\right)^{T} \in \real^{p} \times \real^{q}$. In this case, it is sufficient to require $\sigma$-strong convexity of only one function $f$ or $g$, \ie only in one of the $2$ blocks. Thus, without loss of generality, throughout the rest of this section, the $\sigma$-strong convexity is assumed on the function $g$.
\medskip

	 The notion of \nice primal algorithmic map is flexible and easily adapt to the block setting as stated below.
	\begin{definition}[Nice primal algorithmic map - Block version] \label{D:AlgoMapNiceB}
		Given the parameters $\rho , t > 0$, we let $\tau_{t} = 1$ and $\rho_{t} = \rho$ (when $\sigma = 0$) or $\tau_{t} = t$ and $\rho_{t} = \rho t$ (when $\sigma > 0$). A primal algorithmic map $\SSS_{t} : \real^{n} \times \real^{m} \rightarrow \real^{n}$, which is applied on the augmented Lagrangian $\Lag_{\rho_{t}}\left(z , \lambda\right)$, that generates $z^{+} = \left(u^{+} , v^{+}\right)$ via $z^{+} \in \SSS_{t}\left(z , \lambda\right)$, is called {\em \nice}\hspace{-0.05in}, if there exist a parameter $\delta \in \left(0 , 1\right]$ and matrices $P_{1} , Q_{1} \in \bbS_{+}^{p}$ and $P_{2} , Q_{2} \in \bbS_{+}^{q}$ with $P = \left(P_{1} , P_{2}\right)$ and $Q = \left(Q_{1} , Q_{2}\right)$, such that for any $\left(\xi_{1} , \xi_{2}\right) \in \FFF$ we have
		\begin{align*}
			\Lag_{\rho_{t}}\left(z^{+} , \lambda\right) - \Lag_{\rho_{t}}\left(\xi , \lambda\right) & \leq \frac{1}{t}\Delta_{P_{1}}\left(\xi_{1} , u , u^{+}\right) - \frac{1}{2t}\norm{u^{+} - u}_{Q_{1}}^{2} + \tau_{t}\Delta_{P_{2}}\left(\xi_{2} , v , v^{+}\right) -  \frac{\tau_{t}}{2}\norm{v^{+} - v}_{Q_{2}}^{2} \\
			&  - \frac{\sigma}{2}\norm{\xi_{2} - v^{+}}^{2} - \frac{\delta\rho_{t}}{2}\norm{\AAA z^{+} - b}^{2}.
		\end{align*}				
	\end{definition}	
	The adaptation made in this variant is very simple and appears only in the terms that concern with the parameter $\tau_{t}$. Indeed, recall that the choice of $\tau_{t}$ depends on the convexity or $\sigma$-strong convexity of the given objective function $\Psi\left(\cdot\right)$ through $\sigma$. Therefore, with respect to the block $v$ (which correspond to the possibly strongly convex part) we use $\tau_{t}$ (to include both the convex and strongly convex cases). %On the other hand with respect to the other block, where only convexity is assumed, we fixed $\tau_{t} = 1/t$ (which corresponds to the case where $\sigma = 0$).
\medskip
	
	Furthermore, equipped with Definition \ref{D:AlgoMapNiceB}, it is then straightforward to adapt the two main pillars of the analysis given in Lemmas \ref{L:FGLIter} and \ref{L:FGLMain}, to the block model. For completeness, and for the readers convenience we  record below the two corresponding key lemmas and we omit their proofs since they are essentially identical to the original ones given in Section \ref{SSec:Pillar}.
	\begin{lemma} \label{L:FGLIter-B}
		Let $\left\{ \left(x^{k} , z^{k} , y^{k}\right) \right\}_{k \in \nn}$ be a sequence generated by FLAG. Then, for any $\xi \in \FFF$, $\eta \in \real^{m}$ and $k \geq 0$, we have
		\begin{align}
			\Lag_{\rho_{k}}\left(z^{k + 1} , \eta\right) - \Lag_{\rho_{k}}\left(\xi , \eta\right) & \leq \frac{1}{t_{k}}\Delta_{P_{1}}\left(\xi_{1} , u^{k} , u^{k + 1}\right) + \tau_{k}\Delta_{P_{2}}\left(\xi_{2} , v^{k} , v^{k + 1}\right) - \frac{\sigma}{2}\norm{\xi_{2} - v^{k + 1}}^{2} \nonumber \\
			& + \frac{1}{\mu\rho_{k}}\Delta\left(\eta , y^{k} , y^{k + 1}\right) - \rho t_{k - 1}^{p}\act{\AAA x^{k} - b , \AAA z^{k + 1} - b}. \label{L:FGLIter-B:0}
		\end{align}		
	\end{lemma}
	\begin{lemma} \label{L:FGLMain - B}
		Let $\left\{ \left(x^{k} , z^{k} , y^{k}\right) \right\}_{k \in \nn}$ be a sequence generated by FLAG. Then, for any $\xi \in \FFF$, $\eta \in \real^{m}$ and $k \geq 0$, we have
		\begin{equation} \label{L:FGLMain - B:0}
			t_{k}^{p}s_{k + 1} - t_{k - 1}^{p}s_{k} \leq \frac{\rho_{k}}{t_{k}\rho}\Delta_{P_{1}}\left(\xi_{1} , u^{k} , u^{k + 1}\right) + \frac{\tau_{k}\rho_{k}}{\rho}\Delta_{P_{2}}\left(\xi_{2} , v^{k} , v^{k + 1}\right) - \frac{\rho_{k}\sigma}{2\rho}\norm{\xi_{2} - v^{k + 1}}^{2} + \frac{1}{\mu\rho}\Delta\left(\eta , y^{k} , y^{k + 1}\right),
		\end{equation}
		where $s_{k} = \Lag_{\rho t_{k - 1}^{p}}\left(x^{k} , \eta\right) - \Lag_{\rho t_{k - 1}^{p}}\left(\xi , \eta\right)$ (\cf \ref{laggap}).
	\end{lemma}	
	Equipped with these two lemmas, all our results of Sections \ref{SSec:NonErgodic} and \ref{SSec:Ergodic} can be similarly proved, and therefore all our results applied also for the block model which we briefly summarize as follows:
	\begin{itemize}
		\item If {\em one} of the functions $f$ or $g$ is strongly convex (here, we have arbitrarly chosen $g$) a fast non-ergodic rate can be obtained as in Theorem \ref{T:FGLrate}, and a fast ergodic rate can be obtained as in Corollary \ref{C:FGLrate}.
		\item If {\em none} of the functions $f$ and $g$ is strongly convex, then only classical rate of convergence results can be obtained, namely  a non-ergodic rate as in Theorem \ref{T:GLrate} and an ergodic rate as in Corollary \ref{C:GLrate}.
	\end{itemize}

\subsubsection{Example III: Alternating Direction Method of Multipliers (ADMM)} \label{SSSec:ADMM}
	Thanks to the split nature of the block model it will be worth to exploit this structure by using an Alternating Minimization-based algorithm. The ADMM \cite{GM75, GM1976, FG1983-B} can be obtained from the Augmented Lagrangian method (see Section \ref{SSec:AL}) by applying the Alternating Minimization idea on the primal updating step, which consists of minimizing the augmented Lagrangian jointly with respect to the primal variables $\left(u , v\right)$. In order to use FLAG in this case, all we need is to verify that the corresponding algorithmic map $\left(u^{+} , v^{+}\right) \equiv z^{+} \in \SSS_{t}\left(z , \lambda\right)$ is \nice according to Definition \ref{D:AlgoMapNiceB}, as recorded next without a proof;  instead see the proof of the Linearized ADMM below.
	\begin{lemma}[Proximal ADMM is nice] \label{L:NicePADM}
		Let $M_{1} , M_{2} \succeq 0$, the primal algorithmic map $\SSS_{t}\left(\cdot\right)$ defined by
		\begin{align*}
			u^{+} & = \argmin_{\xi_{1}} \left\{ f\left(\xi_{1}\right) + \act{\lambda , A\xi_{1} + Bv - b} + \frac{\rho_{t}}{2}\norm{A\xi_{1} + Bv - b}^{2} + \frac{1}{2t}\norm{\xi_{1} - u}_{M_{1}}^{2} \right\}, \\ %\label{PADM:Prim1} \\
			v^{+} & = \argmin_{\xi_{2}} \left\{ g\left(\xi_{2}\right) + \act{\lambda , Au^{+} + B\xi_{2} - b} + \frac{\rho_{t}}{2}\norm{Au^{+} + B\xi_{2} - b}^{2} + \frac{\tau_{t}}{2}\norm{\xi_{2} - v}_{M_{2}}^{2} \right\}, %\label{PADM:Prim2}
		\end{align*}				
		is nice with $\delta = 1 - \rho\lambda_{max}(B^{T}B)/\left(\rho\lambda_{max}(B^{T}B) + \lambda_{min}(M_{2})\right)$, $P_{1} = M_{1}$, $P_{2} = M_{2} + \rho B^{T}B$, $Q_{1} = M_{1}$ and $Q_{2} = 0$.
	\end{lemma}

\subsubsection{Example IV: Proximal Linearized Alternating Direction Method of Multipliers} \label{SSSec:LADMM}
	Now, we would like to show that the Proximal Linearized ADMM, which consists of a linearization of the augmented term around the current iteration, also admits a \nice primal algorithmic map as recorded next. It should be noted that here we consider only partial linearization (with respect to the second block $v$), since only this block corresponds to a strongly convex function. Clearly, if we additionally assume that also $f$ is strongly convex then the primal algorithmic map corresponds to the fully Proximal Linearized ADMM, which can be easily proven to be nice in a similar way, hence we omit the details.
	\begin{lemma}[Proximal Linearized ADMM is nice] \label{L:NiceLADM}
		Let $M_{1} , M_{2} \succeq 0$ with $\lambda_{min}(M_{2}) > \rho\lambda_{max}(B^{T}B)$, the primal algorithmic map $\SSS_{t}\left(\cdot\right)$ defined by
		\begin{align}
			u^{+} & = \argmin_{\xi_{1}} \left\{ f\left(\xi_{1}\right) + \act{\lambda , A\xi_{1} + Bv - b} + \frac{\rho_{t}}{2}\norm{A\xi_{1} + Bv - b}^{2} + \frac{1}{2t}\norm{\xi_{1} - u}_{M_{1}}^{2} \right\}, \label{FLADM:Prim1} \\
			v^{+} & = \argmin_{\xi_{2}} \left\{ g\left(\xi_{2}\right) + \act{\lambda , Au^{+} + B\xi_{2} - b} + \rho_{t}\act{Au^{+} + Bv - b , B\xi_{2}} + \frac{\tau_{t}}{2}\norm{\xi_{2} - v}_{M_{2}}^{2} \right\},  \label{FLADM:Prim2}
		\end{align}				
		is nice with $\delta =  1 - \rho\lambda_{max}(B^{T}B)/\lambda_{min}(M_{2})$, $P_{1} = Q_{1} = M_{1}$, $P_{2} = M_{2}$  and $Q_{2} = 0$.
	\end{lemma}
	\begin{proof}
		Applying Lemma \ref{L:ProxIneqLag} with $\varphi\left(\xi\right) = f\left(\xi_{1}\right)$, $c\left(\xi_{1}\right) = \act{\lambda , A\xi_{1}} + \left(\rho_{t}/2\right)\norm{A\xi_{1} + Bv- b}^{2}$ and $W = M_{1}/t$ yields for any $\xi_{1} \in \real^{p}$ that
		\begin{equation*}
			f\left(u^{+}\right) - f\left(\xi_{1}\right) + \act{\lambda + \rho_{t}\left(Au^{+} + Bv - b\right) , Au^{+} - A\xi_{1}} \leq \text{RHS}_{1},
		\end{equation*}
		where
		\begin{equation*}
			\text{RHS}_{1} = \frac{1}{t}\Delta_{M_{1}}\left(\xi_{1} , u , u^{+}\right) - \frac{1}{2t}\norm{u^{+} - u}_{M_{1}}^{2}.
		\end{equation*}
		Similarly, by applying Lemma \ref{L:ProxIneqLag} with $\varphi\left(\xi_{2}\right) = g\left(\xi_{2}\right)$, $c\left(\xi_{2}\right) = \act{\lambda , B\xi_{2}} + \rho_{t}\act{Au^{+} + Bv - b , B\xi_{2}}$ and $W = \tau_{t}M_{2}$ yields for any $\xi_{2} \in \real^{q}$ that
		\begin{equation*}
			g\left(v^{+}\right) - g\left(\xi_{2}\right) + \act{\lambda + \rho_{t}\left(Au^{+} + Bv - b\right) , Bv^{+} - B\xi_{2}} \leq \text{RHS}_{2},
		\end{equation*}
		where
		\begin{equation*}
			\text{RHS}_{2} = \tau_{t}\Delta_{M_{2}}\left(\xi_{2} , v , v^{+}\right) - \frac{\tau_{t}}{2}\norm{v^{+} - v}_{M_{2}}^{2} - \frac{\sigma}{2}\norm{\xi_{2} - v^{+}}^{2}.
		\end{equation*}
		By adding these two inequalities and using the fact that $\xi := \left(\xi_{1} , \xi_{2}\right) \in \FFF$, that is $A\xi_{1} + B\xi_{2} = b$, we get
		\begin{align}
			\Psi\left(z^{+}\right) - \Psi\left(\xi\right) + \act{\lambda , \AAA z^{+} - b} + \rho_{t}\act{Au^{+} + Bv - b , \AAA z^{+} - b} \leq \text{RHS}_{1} + \text{RHS}_{2}, \label{L:NiceLADM:1}
		\end{align}
		where we used that $Au^{+} + Bv^{+} = \AAA z^{+}$ and $z = \left(u , v\right)$, $\xi =(\xi_1,\xi_2)$, $z^{+} = \left(u^{+} , v^{+}\right)$. In addition, we have		
		\begin{align}
			\act{Au^{+} + Bv - b , \AAA z^{+} - b} & = \act{Au^{+} + Bv^{+} - b , \AAA z^{+} - b} + \act{Bv - Bv^{+} , \AAA z^{+} - b} \nonumber \\
			& = \norm{\AAA z^{+} - b}^{2} + \act{Bv - Bv^{+} , \AAA z^{+} - b}  \nonumber \\
			& \geq \norm{\AAA z^{+} - b}^{2} - \frac{1}{2\alpha}\norm{Bv^{+} - Bv}^{2} - \frac{\alpha}{2}\norm{\AAA z^{+} - b}^{2}, \label{L:NiceLADM:2}
		\end{align}
		where we used again that $Au^{+} + Bv^{+} = \AAA z^{+}$, and the simple fact:
		\begin{equation}\label{basic}
		\act{c , d} \geq - \frac{1}{2\alpha} \norm{c}^{2} -  \frac{\alpha}{2} \norm{d}^{2}, \quad \,\, \forall \, c , d \in \real^{n}, \,\, \forall \, \alpha > 0.
		\end{equation}
		By combining \eqref{L:NiceLADM:1} with $\rho_{t}$ times \eqref{L:NiceLADM:2}, we obtain, for all $\alpha < 1$, that
		\begin{align*}
			\Psi\left(z^{+}\right) - \Psi\left(\xi\right) + \act{\lambda , \AAA z^{+} - b}+ \frac{\rho_{t}}{2}\norm{\AAA z^{+} - b}^{2} & \leq \text{RHS}_{1} + \text{RHS}_{2} + \frac{\rho_{t}}{2\alpha}\norm{Bv^{+} - Bv}^{2} - \frac{\delta\rho_{t}}{2}\norm{\AAA z^{+} - b}^{2},
		\end{align*}
		where $\delta = 1 - \alpha > 0$. In order to determine the best parameter $\alpha < 1$ we consider the following terms of the previous inequality (using the definition of $\text{RHS}_{2}$ and omitting the multiplication by $\rho_{t}$):
		\begin{equation*}
			\frac{1}{\alpha}\norm{v^{+} - v}_{B^{T}B}^{2} - \frac{1}{\rho}\norm{v^{+} - v}_{M_{2}}^{2} \leq \left(\frac{1}{\alpha}\lambda_{max}(B^{T}B) - \frac{1}{\rho}\lambda_{min}(M_{2})\right)\norm{v^{+} - v}^{2}.
		\end{equation*}
		Therefore, by taking $\alpha = \rho\lambda_{max}(B^{T}B)/\lambda_{min}(M_{2})$, the desired result is proved.
	\end{proof}		
	\begin{remark} \label{R-zouch}
		As discussed in the introduction, we would like to mention again \cite{LL2019}, where a non-ergodic $O(1/N)$ rate of convergence result for the Linearized ADMM is proven. It should be noted that their proposed algorithm appears very different from our scheme. In the convex case (\ie when $p = 1$), FLAG with the primal algorithmic map $\SSS_{t}\left(\cdot\right)$ chosen to be the Linearized ADMM, produces a scheme sharing the non-ergodic $O(1/N)$ rate of convergence (see Theorem \ref{T:GLrate}).
	\end{remark}
	
\subsubsection{Example V: Chambolle-Pock Method} \label{SSSec:CP}
	The Chambolle-Pock (CP) method \cite{CP2011} is an example of a well-known Lagrangian-based method that was designed to tackle the linear composite model (\cf Example \ref{E:LCM}), which can be seen as a particular case of the block model discuss above with $A = I_{p}$ (where $I_{p}$ is the $p \times p$ identity matrix), \ie
	\begin{equation*}
		\min_{(u , v) \in \real^{n}} \left\{ f\left(u\right) + g\left(v\right) \, : \, u + Bv = b \right\}.
	\end{equation*}
	Following \cite[Section 3.3.]{ST2014}, the CP Method can be seen as special case of the Proximal Linearized ADMM. More precisely, the primal algorithmic map $\SSS_{t}$ that defines the CP method is given by
	\begin{align}
			u^{+} & = \argmin_{\xi_{1}} \left\{ f\left(\xi_{1}\right) + \act{\lambda , \xi_{1} + Bv - b} + \frac{\rho_{t}}{2}\norm{\xi_{1} + Bv - b}^{2} \right\}, \label{FCP:Prim1} \\
			v^{+} & = \argmin_{\xi_{2}} \left\{ g\left(\xi_{2}\right) + \act{\lambda , u^{+} + B\xi_{2} - b} + \rho_{t}\act{u^{+} + Bv - b , B\xi_{2}} + \frac{\tau_{t}}{2\alpha}\norm{\xi_{2} - v}^{2} \right\}, \label{FCP:Prim2}
	\end{align}			
	where $\alpha > 0$ is a proximal parameter (\ie in this case $M_{1} \equiv 0$ and $M_{2} = \left(1/\alpha\right)I_{q}$). Therefore, we immediately obtain from Lemma \ref{L:NiceLADM} that CP also admits a \nice primal algorithmic map as recorded below.
	\begin{lemma}[CP is nice] \label{L:NiceCP}
		Let $\alpha^{-1} > \rho\lambda_{max}(B^{T}B)$, the primal algorithmic map $\SSS_{t}\left(\cdot\right)$ defined by \eqref{FCP:Prim1} and \eqref{FCP:Prim2} is nice with $\delta =  1 - \rho\alpha\lambda_{max}(B^{T}B)$, $P_{1} = Q_{1} = 0$, $P_{2} = \left(1/\alpha\right)I_{q}$  and $Q_{2} = 0$.
	\end{lemma}
	\begin{remark} \label{R:CP1}
		\begin{itemize}
			\item[(a)] Since the primal algorithmic map defined by the CP method is \nice\hspace{-0.05in}, all our results can be applied both in the convex and the strongly convex settings. In the original paper \cite{CP2011}, an ergodic $O(1/N)$ rate was proved, and this can be recovered by our Corollary \ref{C:GLrate}, under the same condition that $\alpha^{-1} > \rho\lambda_{max}(B^{T}B)$.
			\item[(b)] In \cite{CP2016}, a faster ergodic $O(1/N^{2})$ rate was proved in the strongly convex setting. Apparently, it does not seem possible to compare the method proposed there, which also includes other parameters that we do not use here. In any case, here we can use FLAG with the primal iterates as described  in Lemma \ref{L:NiceCP}, and thanks to Corollary \ref{C:FGLrate} with $\sigma\alpha \geq 2$, we obtain a variant of the CP method with the faster ergodic rate of $O(1/N^{2})$.
		\end{itemize}
	\end{remark}
	
\subsubsection{Example VI: Jacobi (Parallel) Direction Method of Multipliers} \label{SSSec:JDMM}
	In continuation of the ADMM presented above, another approach to exploit the block structure is to split them into two parallel steps, which is also known Jacobi Direction Method of Multipliers (JDMM) \cite{HHY2015}. In this case, we have the following result. Note that this algorithm is \nice only when both functions $f$ and $g$ are treated the same, meaning, if they are both strongly convex functions or both are convex functions.
	\begin{lemma}[Proximal Jacobi is nice] \label{L:NiceJDM}
		Let $M_{1} , M_{2} \succeq 0$, the primal algorithmic map $\SSS_{t}\left(\cdot\right)$ defined by
		\begin{align*}
			u^{+} & = \argmin_{\xi_{1}} \left\{ f\left(\xi_{1}\right) + \act{\lambda , A\xi_{1} + Bv - b} + \frac{\rho_{t}}{2}\norm{A\xi_{1} + Bv - b}^{2} + \frac{\tau_{t}}{2}\norm{\xi_{1} - u}_{M_{1}}^{2} \right\}, \\ %\label{PJDM:Prim1} \\
			v^{+} & = \argmin_{\xi_{2}} \left\{ g\left(\xi_{2}\right) + \act{\lambda , Au + B\xi_{2} - b} + \frac{\rho_{t}}{2}\norm{Au + B\xi_{2} - b}^{2} + \frac{\tau_{t}}{2}\norm{\xi_{2} - v}_{M_{2}}^{2} \right\}, %\label{PJDM:Prim2}
		\end{align*}				
		is nice with $\delta = 1 - 2\max\{ \alpha_{1} , \alpha_{2} \}$ where $\alpha_{1} = \rho\lambda_{max}(A^{T}A)/\left(\rho\lambda_{max}(A^{T}A) + \lambda_{min}(M_{1})\right)$ and \\		
		$\alpha_{2} = \rho\lambda_{max}(B^{T}B)/\left(\rho\lambda_{max}(B^{T}B) + \lambda_{min}(M_{2})\right)$, $P_{1} = M_{1} + \rho A^{T}A$, $P_{2} = M_{2} + \rho B^{T}B$ and $Q_{1} = Q_{2} = 0$.
	\end{lemma}
	We omit the proof since it is similar to the more practical scheme described next.

\subsubsection{Example VII: Predictor Corrector Proximal Multiplier Method} \label{SSSec:PCPM}
	In continuation of the JDMM presented above, we would like also to consider the version which consists of linearizing the augmented term. This method, proposed by Chen and Teboulle in \cite{CT1994}, is known in the literature as Predictor Corrector Proximal Multipliers (PCPM). For the sake of completeness, since we did not present the proof of the Proximal JDMM, we prove now that PCPM is \nice \hspace{-0.05in}. As mentioned above, here we should consider that both functions $f$ and $g$ are treated the same. Thus, in this part we assume that $f$ is $\sigma_{f}$-strongly convex function with $\sigma_{f} \geq 0$ and $g$ is $\sigma_{g}$-strongly convex function with $\sigma_{g} \geq 0$, \ie either $\sigma_{f} , \sigma_{g} > 0$ or $\sigma_{f} = \sigma_{g} = 0$.
	\begin{lemma}[PCPM is nice] \label{L:NicePCPM}
		Let $M_{1} , M_{2} \succeq 0$ with $\lambda_{min}(M_{1}) > \rho\lambda_{max}(A^{T}A)$ and $\lambda_{min}(M_{2}) > \rho\lambda_{max}(B^{T}B)$, the primal algorithmic map $\SSS_{t}\left(\cdot\right)$ defined by
		\begin{align}
			u^{+} & = \argmin_{\xi_{1}} \left\{ f\left(\xi_{1}\right) + \act{\lambda , A\xi_{1} + Bv - b} + \rho_{t}\act{Au + Bv - b , A\xi_{1}} + \frac{\tau_{t}}{2}\norm{\xi_{1} - u}_{M_{1}}^{2} \right\}, \label{PCPM:Prim1} \\
			v^{+} & = \argmin_{\xi_{2}} \left\{ g\left(\xi_{2}\right) + \act{\lambda , Au + B\xi_{2} - b} + \rho_{t}\act{Au + Bb - b , B\xi_{2}} + \frac{\tau_{t}}{2}\norm{\xi_{2} - v}_{M_{2}}^{2} \right\}, \label{PCPM:Prim2}
		\end{align}				
		is nice with $\delta = 1 - 2\max\{\alpha_{1} , \alpha_{2}\}$ where $\alpha_{1} = \rho\lambda_{max}(A^{T}A)/\lambda_{min}(M_{1})$ and $\alpha_{2} = \rho\lambda_{max}(B^{T}B)/\lambda_{min}(M_{2})$, $P_{1} = M_{1}$, $P_{2} = M_{2}$ and $Q_{1} = Q_{2} = 0$.
	\end{lemma}
	\begin{proof}
		Applying Lemma \ref{L:ProxIneqLag} with $\varphi\left(\xi\right) = f\left(\xi_{1}\right)$, $c\left(\xi_{1}\right) = \act{\lambda +\rho_{t}\left(Au + Bv- b\right) , A\xi_{1}}$ and $W = \tau_{t}M_{1}$ yields for any $\xi_{1} \in \real^{p}$ that
		\begin{equation*}
			f\left(u^{+}\right) - f\left(\xi_{1}\right) + \act{\lambda + \rho_{t}\left(Au + Bv - b\right) , Au^{+} - A\xi_{1}} \leq \text{RHS}_{1},
		\end{equation*}
		where
		\begin{equation*}
			\text{RHS}_{1} = \tau_{t}\Delta_{M_{1}}\left(\xi_{1} , u , u^{+}\right) - \frac{\tau_{t}}{2}\norm{u^{+} - u}_{M_{1}}^{2} - \frac{\sigma_{f}}{2}\norm{\xi_{1} - u^{+}}^{2}.
		\end{equation*}
		Similarly, by applying Lemma \ref{L:ProxIneqLag} with $\varphi\left(\xi_{2}\right) = g\left(\xi_{2}\right)$, $c\left(\xi_{2}\right) = \act{\lambda + \rho_{t}\left(Au + Bv - b\right) , B\xi_{2}}$ and $W = \tau_{t}M_{2}$ yields for any $\xi \in \real^{q}$ that
		\begin{equation*}
			g\left(v^{+}\right) - g\left(\xi_{2}\right) + \act{\lambda + \rho_{t}\left(Au + Bv - b\right) , Bv^{+} - B\xi_{2}} \leq \text{RHS}_{2},
		\end{equation*}
		where
		\begin{equation*}
			\text{RHS}_{2} = \tau_{t}\Delta_{M_{2}}\left(\xi_{2} , v , v^{+}\right) - \frac{\tau_{t}}{2}\norm{v^{+} - v}_{M_{2}}^{2} - \frac{\sigma_{g}}{2}\norm{\xi_{2} - v^{+}}^{2}.
		\end{equation*}
		By adding these two inequalities and using the fact that $\xi := \left(\xi_{1} , \xi_{2}\right) \in \FFF$, that is $A\xi_{1} + B\xi_{2} = b$, we get
		\begin{align}
			\Psi\left(z^{+}\right) - \Psi\left(\xi\right) + \act{\lambda , \AAA z^{+} - b} + \rho_{t}\act{Au + Bv - b , \AAA z^{+} - b} \leq \text{RHS}_{1} + \text{RHS}_{2}, \label{L:NicePCPM:1}
		\end{align}
		where we used that $Au^{+} + Bv^{+} = \AAA z^{+}$ and $z = \left(u , v\right)$, $\xi := \left(\xi_{1} , \xi_{2}\right)$ , $z^{+} = \left(u^{+} , v^{+}\right)$. Now,  we have		
		\begin{align}
			\act{Au + Bv - b , \AAA z^{+} - b} & = \act{Au^{+} + Bv^{+} - b , \AAA z^{+} - b} + \act{Au + Bv - Au^{+} - Bv^{+} , \AAA z^{+} - b} \nonumber \\
			& = \norm{\AAA z^{+} - b}^{2} + \act{Au - Au^{+} , \AAA z^{+} - b}  + \act{Bv - Bv^{+} , \AAA z^{+} - b}  \nonumber \\
			& \geq \norm{\AAA z^{+} - b}^{2} - \frac{1}{2\alpha_{1}}\norm{Au^{+} - Au}^{2} - \frac{1}{2\alpha_{2}}\norm{Bv^{+} - Bv}^{2} \nonumber \\
			& - \frac{\alpha_{1} + \alpha_{2}}{2}\norm{\AAA z^{+} - b}^{2}, \label{L:NicePCPM:2}
		\end{align}
		where we used again that $Au^{+} + Bv^{+} = \AAA z^{+}$ and the inequality follows by applying \eqref{basic}. Combining \eqref{L:NicePCPM:1} with $\rho_{t}$ times \eqref{L:NicePCPM:2}, we obtain, for all $\alpha < 1$,
		\begin{align*}
			\Psi\left(z^{+}\right) - \Psi\left(\xi\right) + \act{\lambda , \AAA z^{+} - b}+ \frac{\rho_{t}}{2}\norm{\AAA z^{+} - b}^{2} & \leq \text{RHS}_{1} + \frac{\rho_t}{2\alpha_{1}}\norm{Au^{+} - Au}^{2} \\
			&+ \text{RHS}_{2} + \frac{\rho_{t}}{2\alpha_2}\norm{Bv^{+} - Bv}^{2} - \frac{\delta\rho_{t}}{2}\norm{\AAA z^{+} - b}^{2},
		\end{align*}
		where $\delta = 1 - 2\max\{\alpha_{1} , \alpha_{2}\}$ (due to $2\max\{\alpha_{1} , \alpha_{2}\} \geq \alpha_{1} + \alpha_{2}$). In order to determine the best parameter $\alpha_{1} < 1$ we consider the first term of the previous inequality (using the definition of $\text{RHS}_{1}$):
		\begin{equation*}
	\frac{\rho_t}{2} \left \{ \frac{1}{\alpha_{1}}\norm{u^{+} - u}_{A^{T}A}^{2} - \frac{1}{\rho}\norm{u^{+} - u}_{M_{1}}^{2} \right\}  \leq
	\frac{\rho_t}{2} \left \{ \left(\frac{1}{\alpha_{1}}\lambda_{max}(A^{T}A) - \frac{1}{\rho}\lambda_{min}(M_{1})\right)\norm{u^{+} - u}^{2} \right\}.
		\end{equation*}
		Therefore, we will take $\alpha_{1} = \rho\lambda_{max}(A^{T}A)/\lambda_{min}(M_{1})$.  Similarly, we determine $\alpha_{2}$, and the desired result follows.
	\end{proof}		

\subsection{Model (P) with Additional Smooth Function} \label{SSec:Smooth}
	Another situation that might occur in applications, is when the objective function $\Psi$ is given as the sum of two functions $f : \real^{n} \rightarrow \erl$ and $h : \real^{n} \rightarrow \real$, where $f$ is proper, lower semi-continuous and $\sigma$-strongly convex (with $\sigma \geq 0$) while $h$ is convex and continuously differentiable with $L$-Lipschitz continuous gradient, as discussed in Example \ref{E:ASNS}. That is, problems of the following form
	\begin{equation*}
		\min_{x \in \real^{n}} \left\{ f\left(x\right) + h\left(x\right) \, : \, \AAA x = b \right\}.
	\end{equation*}
	
\subsubsection{Example VIII: Proximal Augmented Lagrangian for Composite Objective}
	It is well-known that the best way to tackle $L$-smooth functions in continuous optimization algorithms is via linearization. Therefore, independently of the Lagrangian-based method to be applied on this model, the additional smooth function $h$ can be easily adapted to the FLAG framework. We first record the result for the Proximal Augmented Lagrangian method without a proof; instead see the next result.
	\begin{lemma}[Proximal AL is nice] \label{L:NicePALSmoo}
		Let $M \succeq LI_{n}$, the algorithmic map defined by
		\begin{equation*}
			z^{+} = \SSS_{t}\left(z , \lambda\right) \equiv \argmin_{\xi} \left\{ f\left(\xi\right) + \act{\nabla h\left(z\right) , \xi} + \act{\lambda , \AAA\xi - b} +
			 \frac{\rho_{t}}{2} \| \AAA  \xi - b \|^2 + \frac{\tau_{t}}{2}\norm{\xi - z}_{M}^{2} \right\},
		\end{equation*}
		is nice with $\delta = 1$ and $P = M$ and $Q =  M - LI_{n}$.
	\end{lemma}
	
\subsubsection{Example IX: Proximal Linearized Augmented Lagrangian for Composite Objective}	
	Handling a smooth function in the given objective function is independent of the Lagrangian-based method that is used, and therefore will affect in a similar way (cf. Lemma \ref{L:NicePALSmoo}) any Lagrangian-based method such as ADMM, Proximal Linearized ADMM, CP, Proximal JDMM and PCPM (the value of the corresponding $\delta$ might changed too). For example, if we apply the Linearized Proximal AL method we will also have a \nice primal algorithmic map as shown below.
	\begin{lemma}[Linearized Proximal AL is nice] \label{L:NiceLALSmoo}
		Let $M \succeq \rho\AAA^{T}\AAA + LI_{n}$, the algorithmic map defined by
		\begin{equation*}
			z^{+} = \SSS_{t}\left(z , \lambda\right) \equiv \argmin_{\xi} \left\{ f\left(\xi\right) + \act{\nabla h\left(z\right) , \xi}  + \act{\lambda , \AAA\xi - b} + \rho_{t}\act{\AAA z - b , \AAA\xi} + \frac{\tau_{t}}{2}\norm{\xi - z}_{M}^{2} \right\},
		\end{equation*}
		is nice with $\delta = 1$ and $P = M - \rho\AAA^{T}\AAA$ and $Q = M - \rho\AAA^{T}\AAA - LI_{n}$.
	\end{lemma}
	\begin{proof}
		We apply Lemma \ref{L:ProxIneqLag} on the functions $\varphi\left(\cdot\right) = f\left(\cdot\right)$, $c\left(\xi\right) = \act{\nabla h\left(z\right) , \xi} + \act{\lambda + \rho_{t}\left(\AAA z - b\right) , \AAA\xi}$ and $W = \tau_{t}M$ to obtain, for any $\xi \in \real^{n}$, that
		\begin{align*}
			f\left(z^{+}\right) - f\left(\xi\right) + \act{\lambda + \rho_{t}\left(\AAA z - b\right) , \AAA z^{+} - \AAA\xi} + \act{\nabla h\left(z\right) , z^{+} - \xi} & \leq \tau_{t}\Delta_{M}\left(\xi , z , z^{+}\right) - \frac{\tau_{t}}{2}\norm{z^{+} - z}_{M}^{2} \\
			&  - \frac{\sigma}{2}\norm{\xi - z^{+}}^{2}.
		\end{align*}
		Using the Three-Points Descent Lemma applied on the smooth convex function $h$, we get that
		\begin{equation*}
			h\left(z^{+}\right) - h\left(\xi\right) - \act{\nabla h\left(z\right) , z^{+} - \xi} \leq \frac{L}{2}\norm{z^{+} - z}^{2} \leq \frac{L\rho_{t}}{2\rho}\norm{z^{+} - z}^{2},
		\end{equation*}
		where the last inequality follows from the fact that $\rho_{t}/\rho \geq 1$. Combining these inequalities shows that
		\begin{equation*}
			\Psi\left(z^{+}\right) - \Psi\left(\xi\right) + \act{\lambda , \AAA z^{+} - b} + \rho_{t}\act{\AAA z - b , \AAA z^{+} - \AAA\xi}  \leq \tau_{t}\Delta_{M}\left(\xi , z , z^{+}\right) - \frac{\tau_{t}}{2}\norm{z^{+} - z}_{M - LI_{n}}^{2} - \frac{\sigma}{2}\norm{\xi - z^{+}}^{2}.
		\end{equation*}
		The result now follows using the same arguments as we used in the proof of Lemma \ref{L:NiceLAL}.
	\end{proof}		

\section{Conclusion}
	This paper studies convergence rate analysis of Lagrangian-based methods for convex optimization problems with linear equality constraints. We first introduce the notion of {\em Nice Primal Algorithmic Map}, which plays a central role in the unification and in the simplification of the analysis of most Lagrangian-based methods. Equipped with a \nice primal algorithmic map, we then introduce a simple versatile generic Faster LAGrangian-based (FLAG) method, which allows for the design and analysis of faster schemes with new provably sublinear rate of convergence expressed in terms of functions values and feasibility violation of the {\em original generated sequence}. Classical and fast {\em ergodic} rate of convergence are also obtained as a by-product of our general framework.

\section{Appendix A. Rate of Convergence in the Block Setting}
	For the reader convenience we state and prove Lemmas \ref{L:FGLIter-B} and \ref{L:FGLMain - B}.
	\begin{lemma}
		Let $\left\{ \left(x^{k} , z^{k} , y^{k}\right) \right\}_{k \in \nn}$ be a sequence generated by FLAG. Then, for any $\xi \in \FFF$, $\eta \in \real^{m}$ and $k \geq 0$, we have
		\begin{align*}
			\Lag_{\rho_{k}}\left(z^{k + 1} , \eta\right) - \Lag_{\rho_{k}}\left(\xi , \eta\right) & \leq \frac{1}{t_{k}}\Delta_{P_{1}}\left(\xi_{1} , u^{k} , u^{k + 1}\right) + \tau_{k}\Delta_{P_{2}}\left(\xi_{2} , v^{k} , v^{k + 1}\right) - \frac{\sigma}{2}\norm{\xi_{2} - v^{k + 1}}^{2} \\
			& + \frac{1}{\mu\rho_{k}}\Delta\left(\eta , y^{k} , y^{k + 1}\right) - \rho t_{k - 1}^{p}\act{\AAA x^{k} - b , \AAA z^{k + 1} - b}.
		\end{align*}		
	\end{lemma}
	\begin{proof}
		Since $\left\{ \left(x^{k} , z^{k} , y^{k}\right) \right\}_{k \in \nn}$ is generated by FLAG, we obtain from Definition \ref{D:AlgoMapNiceB} (after omitting the non-negative terms related to $Q_{1}$ and $Q_{2}$), that
		\begin{equation} \label{L:FGLIter-B:1}
			\Lag_{\rho_{k}}\left(z^{k + 1} , \lambda^{k}\right) - \Lag_{\rho_{k}}\left(\xi , \lambda^{k}\right) \leq \frac{1}{t_{k}}\Delta_{P_{1}}\left(\xi_{1} , u^{k} , u^{k + 1}\right) + \tau_{k}\Delta_{P_{2}}\left(\xi_{2} , v^{k} , v^{k + 1}\right) - \frac{\sigma}{2}\norm{\xi_{2} - v^{k + 1}}^{2} - \frac{\delta\rho_{k}}{2}\norm{\AAA z^{k + 1} - b}^{2}.
		\end{equation}
		From the multiplier update \eqref{FGL:MultiStep} and the three-points identity \eqref{Pyth}, we obtain, for all $\eta \in \real^{m}$,
		\begin{align}
			\act{\eta - y^{k} , \AAA z^{k + 1} - b}  = \frac{1}{\mu\rho_{k}}\act{\eta - y^{k} , y^{k + 1} - y^{k}} & = \frac{1}{\mu\rho_{k}}\Delta\left(\eta , y^{k} , y^{k + 1}\right) + \frac{1}{2\mu\rho_{k}}\norm{y^{k + 1} - y^{k}}^{2} \nonumber \\
			& = \frac{1}{\mu\rho_{k}}\Delta\left(\eta , y^{k} , y^{k + 1}\right) + \frac{\mu\rho_{k}}{2}\norm{\AAA z^{k + 1} - b}^{2}. \label{L:FGLIter-B:2}
		\end{align}
		Using the update rule of the sequence $\seq{t}{k}$ (\cf \eqref{FGL:Tk}) we have that $\rho_{k}\left(t_{k} - 1\right) =\rho t_{k - 1}^{p} $ and thus
		\begin{equation*}
			\lambda^{k} = y^{k} + \rho_{k}\left(t_{k} - 1\right)\left(\AAA x^{k} -b\right) = y^{k} + \rho t_{k - 1}^{p}\left(\AAA x^{k} -b\right).
		\end{equation*}
		Using this together with \eqref{L:FGLIter:2} yields, for all $\eta \in \real^{m}$, that
		\begin{align}		
			\act{\eta - \lambda^{k} , \AAA z^{k + 1} - b} & = \act{\eta - y^{k} , \AAA z^{k + 1} - b} - \rho t_{k - 1}^{p}\act{\AAA x^{k} - b , \AAA z^{k + 1} - b} \nonumber \\
			& = \frac{1}{\mu\rho_{k}}\Delta\left(\eta , y^{k} , y^{k + 1}\right) + \frac{\mu\rho_{k}}{2}\norm{\AAA z^{k + 1} - b}^{2} - \rho t_{k - 1}^{p}\act{\AAA x^{k} - b , \AAA z^{k + 1} - b}. \label{L:FGLIter-B:3}
		\end{align}
		Adding \eqref{L:FGLIter-B:1} to \eqref{L:FGLIter-B:3} yields (recall that $\mu \leq \delta$) the desired result.
	\end{proof}	
	\begin{lemma} 
		Let $\left\{ \left(x^{k} , z^{k} , y^{k}\right) \right\}_{k \in \nn}$ be a sequence generated by FLAG. Then, for any $\xi \in \FFF$, $\eta \in \real^{m}$ and $k \geq 0$, we have
		\begin{equation*}
			t_{k}^{p}s_{k + 1} - t_{k - 1}^{p}s_{k} \leq \frac{\rho_{k}}{t_{k}\rho}\Delta_{P_{1}}\left(\xi_{1} , u^{k} , u^{k + 1}\right) + \frac{\tau_{k}\rho_{k}}{\rho}\Delta_{P_{2}}\left(\xi_{2} , v^{k} , v^{k + 1}\right) - \frac{\rho_{k}\sigma}{2\rho}\norm{\xi_{2} - v^{k + 1}}^{2} + \frac{1}{\mu\rho}\Delta\left(\eta , y^{k} , y^{k + 1}\right),
		\end{equation*}
		where $s_{k} = \Lag_{\rho t_{k - 1}^{p}}\left(x^{k} , \eta\right) - \Lag_{\rho t_{k - 1}^{p}}\left(\xi , \eta\right)$ (\cf \ref{laggap}).
	\end{lemma}	
	\begin{proof}
		Identically to the proof of Lemma \ref{L:FGLMain} we have that
		\begin{equation*}
			t_{k}^{p}s_{k + 1} - t_{k - 1}^{p}s_{k} \leq t_{k}^{p - 1}\left(\Lag_{\rho_{k}}\left(z^{k + 1} , \eta\right) - \Lag_{\rho_{k}}\left(\xi , \eta\right)\right) + \rho_{k}t_{k - 1}^{p}\act{\AAA x^{k} - b , \AAA z^{k + 1} - b}.
		\end{equation*}			
		From Lemma \ref{L:FGLIter-B}, after we multiplied both sides by $t_{k}^{p - 1}$ (recall that $\rho_{k} = \rho t_{k}^{p - 1}$), we obtain that
		\begin{align*}
			t_{k}^{p - 1}\left(\Lag_{\rho_{k}}\left(z^{k + 1} , \eta\right) - \Lag_{\rho_{k}}\left(\xi , \eta\right)\right) + \rho_{k}t_{k - 1}^{p}\act{\AAA x^{k} - b , \AAA z^{k + 1} - b} & \leq \frac{\rho_{k}}{t_{k}\rho}\Delta_{P_{1}}\left(\xi_{1} , u^{k} , u^{k + 1}\right) + \frac{\tau_{k}\rho_{k}}{\rho}\Delta_{P_{2}}\left(\xi_{2} , v^{k} , v^{k + 1}\right) \\
			& - \frac{\rho_{k}\sigma}{2\rho}\norm{\xi_{2} - v^{k + 1}}^{2} + \frac{1}{\mu\rho}\Delta\left(\eta , y^{k} , y^{k + 1}\right).
		\end{align*}		
		By combining the last two inequalities, we thus get
		\begin{equation*}
			t_{k}^{p}s_{k + 1} - t_{k - 1}^{p}s_{k} \leq \frac{\rho_{k}}{t_{k}\rho}\Delta_{P_{1}}\left(\xi_{1} , u^{k} , u^{k + 1}\right) + \frac{\tau_{k}\rho_{k}}{\rho}\Delta_{P_{2}}\left(\xi_{2} , v^{k} , v^{k + 1}\right) - \frac{\rho_{k}\sigma}{2\rho}\norm{\xi_{2} - v^{k + 1}}^{2} + \frac{1}{\mu\rho}\Delta\left(\eta , y^{k} , y^{k + 1}\right),
		\end{equation*}				
		which proves the desired result.
	\end{proof}		
	Similarly to the proofs of the non-block case (see Section 4), we prove our main result in the strongly convex setting (\ie $p = 2$) as follows.
	\begin{theorem}[A fast non-ergodic function values and feasibility violation rates] \label{T:FGLrate-B}
		Let $\left\{ \left(x^{k} , z^{k} , y^{k}\right) \right\}_{k \in \nn}$ be a sequence generated by FLAG. Suppose that $\sigma > 0$ and $0 \preceq P_{2} \preceq \left(\sigma/2\right)I_{q}$. Then, for any optimal solution $x^{\ast}$ of problem (P), we have
   		\begin{align}
			\Psi\left(x^{N}\right) - \Psi\left(x^{\ast}\right) & \leq \frac{B_{\rho , c}\left(x^{\ast}\right)}{2N^{2}}, \label{T:FGLrate-B:1} \\
       		\norm{\AAA x^{N} - b} & \leq \frac{B_{\rho , c}\left(x^{\ast}\right)}{cN^{2}}, \label{T:FGLrate-B:2}
		\end{align}
		where $B_{\rho , c}\left(x^{\ast}\right) := 4\left(\norm{x^{\ast} - z^{0}}_{P}^{2} + \frac{1}{\mu\rho}\left(\norm{y^{0}} + c\right)^{2}\right)$.
	\end{theorem}
	\begin{proof}
		From Lemma \ref{L:FGLMain - B}, after we substitute $\xi = x^{\ast} = (u^{\ast} , v^{\ast})^{T}$ and $p = 2$ (recall that $\rho_{k} = \rho t_{k}$ and $\tau_{k} = t_{k}$), we obtain
		\begin{equation*}
			t_{k}^{2}s_{k + 1} - t_{k - 1}^{2}s_{k} \leq \frac{\rho_{k}}{t_{k}\rho}\Delta_{P_{1}}\left(u^{\ast} , u^{k} , u^{k + 1}\right) + \frac{\rho_{k}\tau_{k}}{\rho}\Delta_{P_{2}}\left(v^{\ast} , v^{k} , v^{k + 1}\right) - \frac{\sigma\rho_{k}}{2\rho}\norm{v^{\ast} - v^{k + 1}}^{2} + \frac{1}{\mu\rho}\Delta\left(\eta , y^{k} , y^{k + 1}\right).
		\end{equation*}
		Using the definition of $\Delta_{P_{1}}$, $\Delta_{P_{2}}$ and $\Delta$ (see \eqref{D:Delta}) and the fact that $P_{2} \preceq \left(\sigma/2\right)I_{q}$ we have that
		\begin{align*}
			t_{k}^{2}s_{k + 1} - t_{k - 1}^{2}s_{k} & \leq \frac{1}{2}\left(\norm{u^{\ast} - u^{k}}_{P_{1}}^{2} - \norm{u^{\ast} - u^{k + 1}}_{P_{1}}^{2}\right) + \frac{t_{k}^{2}}{2}\left(\norm{v^{\ast} - v^{k}}_{P_{2}}^{2} - \norm{v^{\ast} - v^{k + 1}}_{P_{2}}^{2}\right) - t_{k}\norm{v^{\ast} - v^{k + 1}}_{P_{2}}^{2} \\
			& + \frac{1}{2\mu\rho}\left(\norm{\eta - y^{k}}^{2} -  \norm{\eta - y^{k + 1}}^{2}\right) \nonumber \\
			& \leq \frac{1}{2}\left(\norm{u^{\ast} - u^{k}}_{P_{1}}^{2} - \norm{u^{\ast} - u^{k + 1}}_{P_{1}}^{2} \right)+ \frac{1}{2}\left(t_{k}^{2}\norm{v^{\ast} - v^{k}}_{P_{2}}^{2} - t_{k + 1}^{2}\norm{v^{\ast} - v^{k + 1}}_{P_{2}}^{2}\right) \\
			& + \frac{1}{2\mu\rho}\left(\norm{\eta - y^{k}}^{2} - \norm{\eta - y^{k + 1}}^{2}\right),
		\end{align*}
		where the last inequality follows from Lemma \ref{L:Tk}(ii). The desired estimates \eqref{T:FGLrate-B:1} and \eqref{T:FGLrate-B:2} are now proved by following the same arguments as in the proof of Theorem \ref{T:FGLrate}.
	\end{proof}		
	Our main result in the convex setting (\ie $\sigma = 0$ and $p = 1$) on the sequence itself is recorded next.
	\begin{theorem}[A non-ergodic function values and feasibility violation rates] \label{T:GLrate-B}
		Let $\left\{ \left(x^{k} , z^{k} , y^{k}\right) \right\}_{k \in \nn}$ be a sequence generated by FLAG and suppose that $\sigma = 0$. Then, for any optimal solution $x^{\ast}$ of problem (P), we have
   		\begin{align}
			\Psi\left(x^{N}\right) - \Psi\left(x^{\ast}\right) & \leq \frac{B_{\rho , c}\left(x^{\ast}\right)}{2N}, \label{T:GLrate:1} \\
       		\norm{\AAA x^{N} - b} & \leq \frac{B_{\rho , c}\left(x^{\ast}\right)}{cN}, \label{T:GLrate:2}
		\end{align}
		where $B_{\rho , c}\left(x^{\ast}\right) := 2\left(\norm{x^{\ast} - z^{0}}_{P}^{2} + \frac{1}{\mu\rho}\left(\norm{y^{0}} + c\right)^{2}\right)$.
	\end{theorem}
	\begin{proof}
		From Lemma \ref{L:FGLMain - B}, after we substitute $\xi = x^{\ast} = (u^{\ast} , v^{\ast})^{T}$, $\sigma = 0$ and $p = 1$ (recall that $\rho_{k} = \rho$ and $\tau_{k} = 1$), we obtain
		\begin{equation*}
			t_{k}s_{k + 1} - t_{k - 1}s_{k} \leq \frac{1}{t_{k}}\Delta_{P_{1}}\left(u_{1}^{\ast} , u^{k} , u^{k + 1}\right) + \Delta_{P_{2}}\left(v^{\ast} , v^{k} , v^{k + 1}\right)  + \frac{1}{\mu\rho}\Delta\left(\eta , y^{k} , y^{k + 1}\right).
		\end{equation*}
		Using the definition of $\Delta_{P_{1}}$, $\Delta_{P_{2}}$ and $\Delta$ (see \eqref{D:Delta}) we have that
		\begin{align*}
			t_{k}s_{k + 1} - t_{k - 1}s_{k} & \leq \frac{1}{2t_{k}}\left(\norm{u^{\ast} - u^{k}}_{P_{1}}^{2} - \norm{u^{\ast} - u^{k + 1}}_{P_{1}}^{2}\right) + \frac{1}{2}\left(\norm{v^{\ast} - v^{k}}_{P_{2}}^{2} - \norm{v^{\ast} - v^{k + 1}}_{P_{2}}^{2}\right) \\
			& + \frac{1}{2\mu\rho}\left(\norm{\eta - y^{k}}^{2} -  \norm{\eta - y^{k + 1}}^{2}\right) \\
			& \leq \frac{1}{2}\left(\frac{1}{t_{k}}\norm{u^{\ast} - u^{k}}_{P_{1}}^{2} - \frac{1}{t_{k + 1}}\norm{u^{\ast} - u^{k + 1}}_{P_{1}}^{2}\right) + \frac{1}{2}\left(\norm{v^{\ast} - v^{k}}_{P_{2}}^{2} - \norm{v^{\ast} - v^{k + 1}}_{P_{2}}^{2}\right) \\
			& + \frac{1}{2\mu\rho}\left(\norm{\eta - y^{k}}^{2} -  \norm{\eta - y^{k + 1}}^{2}\right),
		\end{align*}
		where the second inequality follows from the fact that $-t_{k}^{-1} \leq -t_{k + 1}^{-1}$. Summing this inequality for all $k = 0 , 1 , \ldots , N - 1$ (recall that $t_{-1} = 0$ and $t_{0} =1$) it follows
		\begin{equation*}
			t_{N - 1}\left(\Psi\left(x^{N}\right) - \Psi\left(x^{\ast}\right) + \act{\eta , \AAA x^{N} - b}\right) \leq	t_{N - 1}s_{N} - t_{-1}s_{0} = t_{N - 1}s_{N}\leq \frac{1}{2}\norm{x^{\ast} - z^{0}}_{P}^{2} + \frac{1}{2\mu\rho}\norm{\eta - y^{0}}^{2}.
		\end{equation*}
		Therefore, by taking the maximum of both sides over $\norm{\eta} \leq c$, proves that (recall that $t_{N - 1} = N$ using Lemma \ref{L:Tk}(i))
		\begin{equation*}
			\Psi\left(x^{N}\right) - \Psi\left(x^{\ast}\right) + c\norm{\AAA x^{N} - b} \leq \frac{B_{\rho , c}\left(x^{\ast}\right)}{2N},
		\end{equation*}
		and the desired results follow exactly as done at the end of the proof of Theorem \ref{T:FGLrate}.
	\end{proof}	
\bibliographystyle{plain}
\bibliography{notes}

\end{document}